\documentclass[11pt]{article}

\usepackage{amssymb,amsmath,amsthm,amsfonts}
\usepackage{hyperref, bm, graphicx}

\numberwithin{equation}{section}
\allowdisplaybreaks

\newcommand{\R}{\mathbb{R}}
\newcommand{\Z}{\mathbb{Z}}

\newcommand{\T}{\mathbb{T}}
\newcommand{\A}{\mathcal{A}}
\newcommand{\E}{\mathcal{E}}
\newcommand{\D}{\mathcal{D}}
\newcommand{\bu}{\bm{u}}

\newcommand{\bx}{\bm{x}}
\newcommand{\X}{\bm{X}}
\newcommand{\be}{\bm{e}}
\newcommand{\bg}{\bm{g}}
\newcommand{\bv}{\bm{v}}
\newcommand{\bz}{\bm{z}}
\newcommand{\p}{\partial}
\renewcommand{\div}{{\rm{div}\,}}

\newcommand{\abs}[1]{\left\lvert #1 \right\rvert}
\newcommand{\norm}[1]{\left\lVert #1 \right\rVert}
\newcommand{\wh}[1]{\widehat{#1}}
\newcommand{\wt}[1]{\widetilde{#1}}
\newcommand{\mc}[1]{\mathcal{#1}}

\newtheorem{theorem}{Theorem}[section]
\newtheorem{lemma}[theorem]{Lemma}
\newtheorem{proposition}[theorem]{Proposition}
\newtheorem{corollary}[theorem]{Corollary}
\newtheorem{definition}[theorem]{Definition}
\theoremstyle{definition}
\newtheorem{remark}[theorem]{Remark}

\begin{document}
\title{An error bound for the slender body approximation of a thin, rigid fiber sedimenting in Stokes flow}
\author{Yoichiro Mori, Laurel Ohm
\footnote{This research was supported in part by NSF grant DMS-1620316, DMS-1516978, and DMS-1907583, awarded to Y.M., and a Torske Klubben Fellowship, awarded to L.O. L.O. also thanks Dallas Albritton for helpful discussion. }\\ 
\textit{\small School of Mathematics, University of Minnesota, Minneapolis, MN 55455}}
\date{}

\maketitle

\begin{abstract}
We investigate the motion of a thin rigid body in Stokes flow and the corresponding slender body approximation used to model sedimenting fibers. In particular, we derive a rigorous error bound comparing a regularized version of the rigid slender body approximation to the classical PDE for rigid motion in the case of a closed loop with constant radius. Our main tool is the slender body PDE framework established by the authors and D. Spirn in \cite{closed_loop,free_ends}, which we adapt to the rigid setting.
\end{abstract}

\section{Introduction}
Determining the motion of a three-dimensional rigid body sedimenting in a Stokesian fluid is an important problem in both theoretical and computational fluid mechanics. This motion is described by a classical PDE \cite{corona2017integral,galdi1999steady,weinberger1972variational}, which we write below in the case of a thin rigid body. We use $\E(\bu) = \frac{1}{2}(\nabla\bu+(\nabla\bu)^{\rm T})$ to denote the symmetric gradient, and $\bm{\sigma}=\bm{\sigma}(\bu,p) = 2\E(\bu)-p{\bf I}$ to denote the stress tensor. Let $\Sigma_\epsilon$ denote a closed loop slender body of radius $\epsilon>0$ (to be made precise in Section \ref{geometry}) and let $\Omega_\epsilon =\R^3 \setminus \overline{\Sigma_\epsilon}$ and $\Gamma_\epsilon=\p\Sigma_\epsilon$ (see Figure \ref{fig:coord_sys}). For simplicity, we take the center of mass of the body to be at the origin. The full PDE description of a slender body undergoing a rigid motion in Stokes flow may be written as follows:
\begin{equation}\label{rigid}
\begin{aligned}
-\Delta \bu^{\rm r} +\nabla p^{\rm r} &=0   \hspace{2.6cm} \text{ in } \Omega_\epsilon \\
\div \bu^{\rm r} &= 0  \hspace{2.6cm} \text{ in } \Omega_\epsilon \\
\bu^{\rm r}(\bx) &= \bv^{\rm r} + \bm{\omega}^{\rm r}\times \bx, \qquad \bx \in \Gamma_\epsilon \\
\bu^{\rm r}(\bx) &\to 0  \hspace{2.6cm} \text{as }\abs{\bx}\to \infty 
\end{aligned}
\end{equation}
and
\begin{align*}
\int_{\Gamma_\epsilon} \bm{\sigma}^{\rm r}\bm{n} \; dS &= \bm{F}, \quad \int_{\Gamma_\epsilon} \bx \times (\bm{\sigma}^{\rm r}\bm{n}) \; dS = \bm{T} .
\end{align*}
We are interested in the \emph{mobility problem} \cite{corona2017integral}, where the total force $\bm{F}\in \R^3$ and torque $\bm{T}\in \R^3$ are given and we solve for the linear velocity $\bv^{\rm r}\in \R^3$ and angular velocity $\bm{\omega}^{\rm r}\in\R^3$ of the body. Note that the boundary value problem \eqref{rigid} is in fact valid for a rigid body of arbitrary shape, but for the purposes of this paper we specifically consider here a slender closed loop. Using the variational framework of \cite{galdi1999steady,gonzalez2004dynamics,weinberger1972variational}, it can be shown that \eqref{rigid} is a well-posed PDE.  \\

On the computational side, there has been much recent interest in numerical simulations of rigid particle sedimentation \cite{guazzelli2006sedimentation, guazzelli2011fluctuations}, and various tools have been developed to facilitate these simulations \cite{corona2017integral, jung2006periodic, mitchell2015sedimentation}. \\


\begin{figure}[!h]
\centering
\includegraphics[scale=0.6]{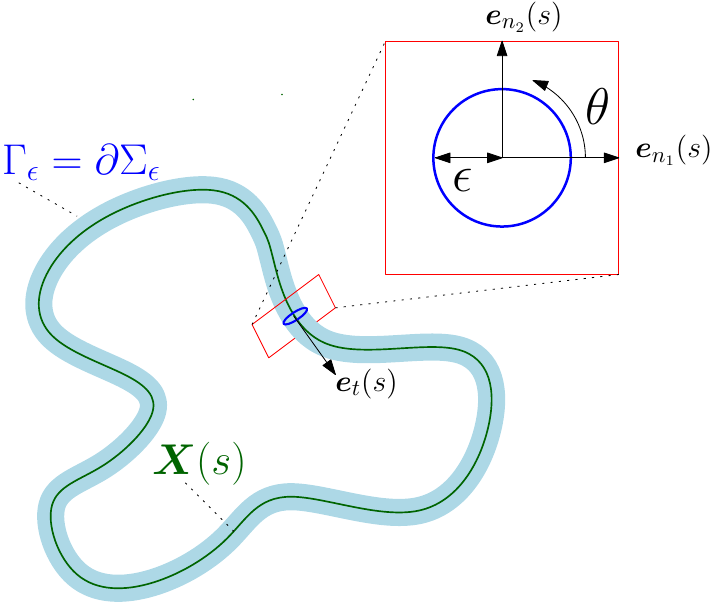}\\
\caption{The geometry of the rigid fiber may be parameterized with respect to the orthogonal frame $\be_t(s)$, $\be_{n_1}(s)$, $\be_{n_2}(s)$ defined in Section \ref{geometry}.}
\label{fig:coord_sys}
\end{figure}

For a thin rigid body, a commonly-used tool for simplifying simulations is slender body theory, which exploits the thin geometry of the body by approximating the filament as a one-dimensional force density distributed along the fiber centerline. Slender body theory is a popular method for modeling sedimentation of thin fibers, both rigid \cite{ butler2002dynamic, park2010cloud, saintillan2005smooth, shin2009structure} and semi-flexible \cite{li2013sedimentation,manikantan2014instability}. Here we will specifically consider the slender body theory established by Keller and Rubinow \cite{keller1976slender} and further developed in \cite{gotz2000interactions, johnson1980improved, tornberg2004simulating}. This slender body theory is derived by integrating the fundamental solution to the Stokes equations (the Stokeslet) and higher order corrections along the fiber centerline, yielding an expression valid only away from the fiber centerline. A limiting expression valid on the centerline itself \eqref{SBT_expr} is then obtained via a matched asymptotic expansion.\\

Let $\X : \T\equiv \R / \Z \to \R^3$ denote the coordinates of the slender body centerline, parameterized by arclength $s$ and defined more precisely in Section \ref{geometry}. Given a line force density $\bm{f}(s)$, $s\in \T$, the slender body approximation yields a direct expression approximating the velocity of the fiber, given by \cite{shelley2000stokesian}:
\begin{equation}\label{SBT_expr}
\begin{aligned}
\bu^{\rm s}_{\rm C}(s) &= \bm{\Lambda}[\bm{f}](s) + \bm{K}[\bm{f}](s),\\
\bm{\Lambda}[\bm{f}](s) &:=  \frac{1}{8\pi}\big[({\bf I}- 3\be_t\be_t^{\rm T})-2({\bf I}+\be_t\be_t^{\rm T}) \log(\pi\epsilon/4) \big]{\bm f}(s) \\
\bm{K}[\bm{f}](s) &:= \frac{1}{8\pi}\int_{\T} \left[ \left(\frac{{\bf I}}{|\bm{R}_0|}+ \frac{\bm{R}_0\bm{R}_0^{\rm T}}{|\bm{R}_0|^3}\right){\bm f}(s') - \frac{{\bf I}+\be_t(s)\be_t(s)^{\rm T} }{|\sin (\pi(s-s'))/\pi|} {\bm f}(s)\right] \, ds'.
\end{aligned}
\end{equation}
Here $\be_t(s)$ is the unit tangent vector to $\X(s)$ and $\bm{R}_0(s,s') = \X(s) - \X(s')$. The slender body approximation generally allows for bending and flexing of the filament along its centerline and requires specifying the one-dimensional force density over the length of the fiber centerline. If the fiber is constrained to be fully rigid, only the total force $\bm{F}$ and torque $\bm{T}$ must be specified, where 
\begin{equation}\label{SBT_cond1}
 \int_{\T} \bm{f}(s) \, ds = \bm{F}, \quad \int_{\T} \X(s)\times \bm{f}(s) \, ds = \bm{T}.
 \end{equation}
These constraints give rise to a system of integral equations which must be solved to obtain the line force density along the slender body (see \cite{gustavsson2009gravity, tornberg2006numerical}; also \cite{ butler2002dynamic, park2010cloud, saintillan2005smooth, shin2009structure}). However, these integral equations are ill-posed. Specifically, a detailed spectral analysis by G\"otz \cite{gotz2000interactions} in the case of a straight slender body centerline shows that the slender body operator $(\bm{\Lambda}+\bm{K})$ in \eqref{SBT_expr} is not invertible for all small $\epsilon$. A similar result for a perfectly circular, planar centerline was shown by Shelley-Ueda in \cite{shelley2000stokesian}. For fibers with more general centerline curvature, a spectral analysis of the slender body integral operator is complicated, but it is expected that the ill-posedness remains.  \\

In practice, this ill-posedness is addressed by regularizing the integral operator $\bm{K}$ to remove the invertibility issues encountered with $(\bm{\Lambda}+\bm{K})$. Various regularizations are possible; see Section \ref{regulars} for one example. Therefore, to analyze the error in the rigid slender body approximation, we will instead consider the regularized expression 
\begin{equation}\label{SBT_reg}
\bu^{\rm s}_\text{reg}(s) = \bm{\Lambda}[\bm{f}](s) + \bm{K}[\bm{f}](s) + r_\epsilon[\bm{f}](s),
\end{equation}
where $r_\epsilon$ is a regularization such that the operator $(\bm{\Lambda}+\bm{K}+r_\epsilon)$ \emph{is} invertible. The regularization $r_\epsilon$ may be effected, for example, by regularizing the integral kernel $\bm{K}$: 
\begin{equation}\label{Kreg0} 
\begin{aligned}
r_\epsilon  &= \bm{K}_{\rm reg} - \bm{K}; \\
\bm{K}_\text{reg}[\bm{f}^{\rm s}](s)  &=  \frac{1}{4\pi}\big[\be_t\be_t^{\rm T}+({\bf I}+\be_t\be_t^{\rm T})\log(\pi\epsilon/4) \big]{\bm f}^{\rm s}(s) \\
&\quad + \frac{1}{8\pi}\int_{\T} \left(\frac{{\bf I}}{(|\bm{R}_0|^2+\epsilon^2)^{1/2}}+ \frac{\bm{R}_0\bm{R}_0^{\rm T}}{|\bm{R}_0|^2(|\bm{R}_0|^2+\epsilon^2)^{1/2}}\right){\bm f}^{\rm s}(s') \, ds'.
\end{aligned}
\end{equation}
This choice of regularization is explored in more detail in Section \ref{regulars}. Various other regularizations are possible, including numerical truncation of the continuous integral operator $\bm{K}$. To include room for other possible regularizations, we leave the particular form of $r_\epsilon$ unspecified for much of the analysis. However the regularization is chosen, the idea is that $r_\epsilon$ should be small in terms of $\epsilon$ so that \eqref{SBT_reg} is close to the expression \eqref{SBT_expr}. In particular, the $\bm{K}_{\rm reg}$ example explored in Section \ref{regulars} satisfies $\big\| r_\epsilon[\bm{f}^{\rm s}] \big\|_{L^2(\T)}\le C\epsilon\abs{\log\epsilon}\|\bm{f}^{\rm s}\|_{C^1(\T)}$. \\

Combined with the conditions \eqref{SBT_cond1} and the constraint that the velocity of the slender body centerline is a rigid motion, i.e.
\begin{equation}\label{SBT_cond2}
\bu^{\rm s}_\text{reg}(s) = \bv^{\rm s}+\bm{\omega}^{\rm s}\times \X(s), \quad  \bv^{\rm s}, \bm{\omega}^{\rm s} \in \R^3,
\end{equation}
 the regularized expression \eqref{SBT_reg} with an appropriate choice of $r_\epsilon$ likely gives rise to a well-posed rigid slender body approximation. A general solution theory for such regularized equations will require a detailed spectral analysis of the regularized integral operator, which is beyond the scope of this paper. \\


Instead, the aim of this paper is to establish an {\it a posteriori} error bound between the regularized slender body approximation for rigid motion in Stokes flow \eqref{SBT_cond1}-\eqref{SBT_cond2} and the classical PDE \eqref{rigid}. We take for granted that the regularized slender body approximation \eqref{SBT_reg} gives rise to $\bm{f}^{\rm s}\in C^1(\T)$ satisfying \eqref{SBT_cond1} and \eqref{SBT_cond2}. This $\bm{f}^{\rm s}$ must then appear in the final error bound, giving rise to a type of {\it a posteriori} error estimate, similar to the type of estimates commonly used in finite element analysis \cite{ainsworth2011posteriori}. To obtain an {\it a priori} bound, we would need a general solution theory for  \eqref{SBT_cond1}-\eqref{SBT_cond2} to then able to say that such an $\bm{f}^{\rm s}$ is then bounded by the given $\bm{F}$ and $\bm{T}$. We show the following theorem. 
\begin{theorem}\label{rigid_theorem}
Let $\Sigma_\epsilon$ be a slender body as defined in Section \ref{geometry}. Suppose the total force $\bm{F}\in \R^3$ and torque $\bm{T}\in \R^3$ are given, and assume that regularized rigid slender body approximation \eqref{SBT_cond1}-\eqref{SBT_cond2} is satisfied by some $\bm{f}^{\rm s}\in C^1(\T)$. Then the difference $\bv^{\rm r}-\bv^{\rm s}$, $\bm{\omega}^{\rm r}-\bm{\omega}^{\rm s}$ between the linear and angular velocities of true rigid motion \eqref{rigid} and the regularized slender body approximation \eqref{SBT_cond1}-\eqref{SBT_cond2} satisfies 
\begin{equation}\label{rigid_err}
\abs{\bv^{\rm r}-\bv^{\rm s}} + \abs{\bm{\omega}^{\rm r}-\bm{\omega}^{\rm s}} \le C\bigg(\sqrt{\epsilon}\abs{\log\epsilon}^{3/2}\big(\norm{\bm{f}^{\rm s}}_{C^1(\T)}+ \abs{\bm{F}}+\abs{\bm{T}}\big) + \epsilon^{-1/2}\big\| r_\epsilon[\bm{f}^{\rm s}] \big\|_{L^2(\T)}\bigg)
\end{equation}
for $C$ depending on $c_\Gamma$, $\kappa_{\max}$, and $\xi_{\max}$. 
\end{theorem}
The constants $c_\Gamma$, $\kappa_{\max}$, and $\xi_{\max}$ have to do only with the shape of the fiber centerline and are defined in Section \ref{geometry}. Note that in order to obtain a convergence result, the regularization $r_\epsilon$ must be chosen to be sufficiently small, e.g. such that $\big\| r_\epsilon[\bm{f}^{\rm s}] \big\|_{L^2(\T)}\le C\epsilon\abs{\log\epsilon}\|\bm{f}^{\rm s}\|_{C^1(\T)}$ (see Section \ref{regulars}). It may be possible to improve the $\sqrt{\epsilon}$ bound given a more complete solution theory for the slender body approximation.  \\

In order to prove Theorem \ref{rigid_theorem}, we introduce an intermediary PDE which we will call the {\it slender body PDE for rigid motion}. The idea follows from the notion of slender body PDE proposed by the authors and D. Spirn in \cite{closed_loop} and \cite{free_ends} as a framework for analyzing the error introduced by the Keller-Rubinow slender body approximation for closed-loop and open-ended fibers, respectively. To construct the rigid slender body PDE, we impose that the velocity of the slender body is uniform over each cross section $s$ of the fiber. In particular, we approximate $\bx \in \Gamma_\epsilon$ as its $L^2$ projection onto the fiber centerline $\X(s)$, thereby ignoring slight differences in torque across the slender body. Note that the slender body geometry is defined in Section \ref{geometry} such that this projection onto the fiber centerline is unique; i.e. the notion of ``fiber cross section" is well-defined. We define the slender body PDE for rigid motion as follows:
\begin{equation}\label{SB_PDE}
\begin{aligned}
-\Delta \bu^{\rm p} +\nabla p^{\rm p} &=0  \hspace{3.27cm} \text{in } \Omega_\epsilon  \\
 \div \bu^{\rm p} &=0  \hspace{3.27cm} \text{in } \Omega_\epsilon  \\
\bu^{\rm p}(\bx) &= \bv^{\rm p} + \bm{\omega}^{\rm p}\times \X(s) \qquad \text{on } \Gamma_\epsilon  \\
\bu^{\rm p}(\bx) &\to 0 \hspace{3.23cm} \text{as }\abs{\bx}\to \infty
\end{aligned}
\end{equation}
and
\begin{align*}
\int_{\Gamma_\epsilon} \bm{\sigma}^{\rm p}\bm{n} \; \mc{J}_\epsilon(s,\theta) \, d\theta \, ds &= \bm{F}, \qquad
\int_{\T} \X(s)\times \bigg(\int_0^{2\pi} \bm{\sigma}^{\rm p}\bm{n} \, \mc{J}_\epsilon(s,\theta) d\theta \bigg) ds = \bm{T}.
\end{align*}
Here we have written $dS = \mc{J}_\epsilon(s,\theta)\, d\theta\, ds$, where $\mc{J}_\epsilon$ is the Jacobian factor on the slender body surface, which we parameterize as a tube about $\X(s)$ using surface angle $\theta$ (see Section \ref{geometry} and expression \eqref{jac_fac}). We show that for a closed filament, the rigid slender body PDE is in fact close to the classical PDE for rigid motion \cite{galdi1999steady,weinberger1972variational} -- in particular, the variation in torque over any cross section of the slender body is higher order in $\epsilon$. \\

In the case of a flexible filament with a prescribed force density per unit length along the centerline, the slender body PDE of \cite{closed_loop,free_ends} is well-posed, and the difference between the slender body approximation and the PDE solution can be estimated in terms of the slender body radius and the given line force density. We aim to use the existing error analysis in \cite{closed_loop} to bound the difference between the rigid slender body approximation and the rigid slender body PDE solution. The rigid case is complicated by the fact that the existing error bound relies on knowledge of the line force density along the filament, while only $\bm{F}$ and $\bm{T}$ are specified. Below we outline our treatment of this and other complications arising in the proof of Theorem \ref{rigid_theorem}.


\subsection{Outline of the proof of Theorem \ref{rigid_theorem}}
The strategy for proving Theorem \ref{rigid_theorem} is to show that, given $\bm{F}$ and $\bm{T}$, the solution to the rigid slender body PDE \eqref{SB_PDE} is close to both the classical rigid PDE solution \eqref{rigid} and the rigid slender body approximation \eqref{SBT_cond1} - \eqref{SBT_cond2}. \\

First, we must show that the rigid slender body PDE is well-posed. Using Definition \ref{rigid_weak_p} of a weak solution to the rigid slender body PDE \eqref{SB_PDE}, where the function space $\mc{R}_\epsilon^{\div}$ is as defined after \eqref{Repsilon}, we show the following. 
\begin{theorem}\label{SB_PDE_well}
Let $\Sigma_\epsilon$ be a slender body as defined in Section \ref{geometry}. Given $\bm{F}$ and $\bm{T}\in\R^3$, there exists a unique weak solution $(\bu^{\rm p},p^{\rm p})\in \mc{R}_\epsilon^{\div} \times L^2(\Omega_\epsilon)$ to the slender body PDE for rigid motion \eqref{SB_PDE} satisfying the estimate
\begin{equation}\label{totalPest}
\norm{\nabla\bu^{\rm p}}_{L^2(\Omega_\epsilon)} + \norm{p^{\rm p}}_{L^2(\Omega_\epsilon)} \le C\abs{\log\epsilon}^{1/2} ( \abs{\bm{F}}+\abs{\bm{T}} )
\end{equation}
for $C$ depending on $c_\Gamma$ and $\kappa_{\max}$.
\end{theorem}
Theorem \ref{SB_PDE_well} can be established using many of the same tools from the well-posedness theory in \cite{closed_loop}. In addition, we will make use of the following bound along the slender body centerline $\X(s)$:
 \begin{lemma}\label{v_omega_bd}
 Let $\X$ be as in Section \ref{geometry} and consider constant vectors $\bv$, $\bm{\omega} \in \R^3$. Then
 \begin{equation}
 \abs{\bv}+\abs{\bm{\omega}} \le C\norm{\bv+ \bm{\omega}\times \X}_{L^2(\T)}
 \end{equation}
 for $C$ depending only on $c_\Gamma$ and $\kappa_{\max}$. 
 \end{lemma}
 We will first prove Lemma \ref{v_omega_bd} in Section \ref{2ndLemma}; then Theorem \ref{SB_PDE_well} quickly follows using some of the key inequalities collected in Section \ref{ineq}. \\

With the variational framework for \eqref{SB_PDE}, comparing \eqref{rigid} to \eqref{SB_PDE} is relatively straightforward. Using Lemma \ref{v_omega_bd}, we show that the difference between the true rigid motion \eqref{rigid} and the slender body PDE description \eqref{SB_PDE} satisfies the following lemma.
\begin{lemma}\label{true_vs_SB}
Let $\X$ be as in Section \ref{geometry}. Given $\bm{F}$ and $\bm{T}\in \R^3$, let $(\bv^{\rm r},\bm{\omega}^{\rm r})$ be the corresponding boundary values satisfying \eqref{rigid} and let $(\bv^{\rm p},\bm{\omega}^{\rm p})$ be the boundary values satisfying \eqref{SB_PDE}. Then
\begin{equation}\label{true_vs_SB_bd}
\abs{\bm{\omega}^{\rm r}-\bm{\omega}^{\rm p}}+ \abs{\bv^{\rm r}-\bv^{\rm p}} \le \epsilon\abs{\log\epsilon}C( \abs{\bm{T}}+ \abs{\bm{F}})
\end{equation}
where $C$ depends only on $c_\Gamma$ and $\kappa_{\max}$. 
\end{lemma}


The main difficulties in proving Theorem \ref{rigid_theorem} arise in comparing \eqref{SB_PDE} to \eqref{SBT_cond1} - \eqref{SBT_cond2}. As discussed, we assume that we are considering a rigid slender body approximation \eqref{SBT_cond1} - \eqref{SBT_cond2} that gives rise to a force density $\bm{f}^{\rm s}\in C^1(\T)$. A difficulty is that in order to use the error analysis framework of \cite{closed_loop}, the line force density along the slender body must be the same for both the slender body approximation and the slender body PDE. Therefore we need to define yet another intermediary PDE. \\

Given $\bm{f}^{\rm s}\in C^1(\T)$ satisfying \eqref{SBT_cond1} - \eqref{SBT_cond2} for given $\bm{F}$ and $\bm{T}\in \R^3$, we define $\bu^{\rm p,s}$ as the solution to the PDE: 
\begin{equation}\label{SB_PDE_ps}
\begin{aligned}
-\Delta \bu^{\rm p,s} +\nabla p^{\rm p,s} &=0  \hspace{2.75cm} \text{in } \Omega_\epsilon  \\
 \div \bu^{\rm p,s} &=0 \hspace{2.75cm} \text{in } \Omega_\epsilon  \\
\int_0^{2\pi} (\bm{\sigma}^{\rm p,s}\bm{n}) \; \mc{J}_\epsilon(s,\theta) \, d\theta &= \bm{f}^{\rm s}(s)  \hspace{2.05cm} \text{on }\Gamma_\epsilon \\
\bu^{\rm p,s}\big|_{\Gamma_\epsilon} &= {\rm Tr}(\bu^{\rm p,s})(s), \qquad \text{ unknown but independent of } \theta \\
\bu^{\rm p,s}&\to 0  \hspace{2.65cm} \text{as } \abs{\bx}\to\infty.
\end{aligned}
\end{equation}
Here ${\rm Tr}(\bu^{\rm p,s})(s)$ denotes the trace of $\bu^{\rm p,s}$ on $\Gamma_\epsilon$, and $\theta$ refers to the parameterization of $\Gamma_\epsilon$ as a tube about $\X(s)$ (see Section \ref{geometry}). By \cite{closed_loop}, we know that a (weak) solution $(\bu^{\rm p,s},p^{\rm p,s})$ exists and is unique. Now, ${\rm Tr}(\bu^{\rm p,s})(s)$ may not be precisely a rigid motion, but we can show that it is close. In particular, by Theorem 1.3 in \cite{closed_loop}, we may bound the difference between ${\rm Tr}(\bu^{\rm p,s})(s)$ and the non-regularized slender body approximation $\bu^{\rm s}_{\rm C}(s)$ \eqref{SBT_expr} by
\begin{equation}\label{noreg_err} 
\norm{{\rm Tr}( \bu^{\rm p,s}) - \bu^{\rm s}_{\rm C} }_{L^2(\T)} \le C\epsilon \abs{\log\epsilon}^{3/2}\norm{\bm{f}^{\rm s}}_{C^1(\T)} 
\end{equation}
for $C$ depending only on $c_\Gamma$ and $\kappa_{\max}$. The regularized slender body approximation $\bu^{\rm s}_{\text{reg}}(s)= \bu^{\rm s}_{\rm C}(s) - r_\epsilon[\bm{f}^{\rm s}](s) = \bv^{\rm s}+\bm{\omega}^{\rm s} \times \X(s)$ \eqref{SBT_reg} then satisfies
\begin{equation}\label{centerline_err} 
\norm{{\rm Tr}( \bu^{\rm p,s}) - (\bv^{\rm s}+\bm{\omega}^{\rm s} \times \X) }_{L^2(\T)} \le C\epsilon \abs{\log\epsilon}^{3/2}\norm{\bm{f}^{\rm s}}_{C^1(\T)} + \big\| r_\epsilon[\bm{f}^{\rm s}] \big\|_{L^2(\T)}.
\end{equation}

A further technical issue arises in comparing \eqref{SB_PDE_ps} to \eqref{SB_PDE}. In order to obtain a useful estimate of the difference between $(\bu^{\rm p,s}-\bu^{\rm p},p^{\rm p,s}-p^{\rm p})$ in terms of only $\bm{F}$, $\bm{T}$, and $\bm{f}^{\rm s}(s)$, we will need a careful characterization of the $\epsilon$-dependence in a higher regularity estimate for solutions to \eqref{SB_PDE} (see Lemma \ref{high_reg}). Note that for a (sufficiently smooth) sedimenting rigid body, once well-posedness of the PDE has been established, higher regularity of the solution follows by standard arguments for a Stokes Dirichlet boundary value problem. In our case, the novelty is determining how the higher regularity bound scales with $\epsilon$. Our proof (see Appendix \ref{reg_lem}) makes use of the local coordinate system valid near the slender body. We obtain commutator estimates for the tangential derivatives along the slender body surface and use an integration by parts argument, along with the form of the Stokes equations in local coordinates, to show that the bound for an additional derivative of the rigid slender body PDE solution scales like $1/\epsilon$, up to logarithmic corrections. \\ 

Now, using the variational framework for the slender body PDE along with this higher regularity lemma, we can show the following estimate.
\begin{lemma}\label{ups_up_err}
Let $\bu^{\rm p,s}$ satisfy \eqref{SB_PDE_ps} and let $(\bv^{\rm p},\bm{\omega}^{\rm p})$ denote the rigid slender body PDE boundary values satisfying \eqref{SB_PDE}. Then 
\begin{equation}
\|{\rm Tr}( \bu^{\rm p,s}) - ( \bv^{\rm p} + \bm{\omega}^{\rm p}\times \X)\|_{L^2(\T)} \le C\bigg(\sqrt{\epsilon}\abs{\log\epsilon}^{3/2}\big(\norm{\bm{f}^{\rm s}}_{C^1(\T)}+ \abs{\bm{F}}+\abs{\bm{T}}\big) + \epsilon^{-1/2}\big\| r_\epsilon[\bm{f}^{\rm s}] \big\|_{L^2(\T)}\bigg)
\end{equation}
for $C$ depending on $c_\Gamma$, $\kappa_{\max}$, and $\xi_{\max}$. 
\end{lemma}

Combining estimate \eqref{centerline_err} with Lemma \ref{ups_up_err} and using Lemma \ref{v_omega_bd} with $\bv^{\rm p}-\bv^s$ and $\bm{\omega}^{\rm p}-\bm{\omega}^{\rm s}$ in place of $\bv$ and $\bm{\omega}$, we obtain the following bound for the difference between the regularized slender body approximation \eqref{SBT_cond1}-\eqref{SBT_cond2} and the slender body PDE \eqref{SB_PDE}: 
\begin{equation}\label{PDE_vs_SBT}
\abs{\bv^{\rm p}-\bv^{\rm s}} + \abs{\bm{\omega}^{\rm p}-\bm{\omega}^{\rm s}} \le C\bigg(\sqrt{\epsilon}\abs{\log\epsilon}^{3/2}\big(\norm{\bm{f}^{\rm s}}_{C^1(\T)}+ \abs{\bm{F}}+\abs{\bm{T}}\big) + \epsilon^{-1/2}\big\| r_\epsilon[\bm{f}^{\rm s}] \big\|_{L^2(\T)}\bigg).
\end{equation}

Finally, combining the estimate \eqref{PDE_vs_SBT} with Lemma \ref{true_vs_SB} yields Theorem \ref{rigid_theorem}. The remainder of this paper is thus devoted to showing Lemmas \ref{v_omega_bd} - \ref{ups_up_err}. We will begin by introducing the variational framework for \eqref{SB_PDE} and noting some key inequalities in Section \ref{variational0}. In Section \ref{2ndLemma}, we show Lemma \ref{v_omega_bd} and use it to derive estimates for $(\bu^{\rm p},p^{\rm p},\bv^{\rm p},\bm{\omega}^{\rm p})$ satisfying \eqref{SB_PDE}. These estimates can then be used to show Theorem \ref{SB_PDE_well}. In Section \ref{3rdLemma}, we use the variational framework for the rigid slender body PDE to prove Lemma \ref{true_vs_SB}. Finally, in Section \ref{1stLemma}, we prove Lemma \ref{ups_up_err} to complete the proof of Theorem \ref{rigid_theorem}.  


\section{Geometry and variational framework}\label{variational0}
We begin in Section \ref{geometry} with a precise definition of the slender body geometry. In Section \ref{variational}, we introduce the variational form of the slender body PDE for rigid motion \eqref{SB_PDE}, which, along with the variational form of \eqref{SB_PDE_ps}, will provide the framework for obtaining Theorem \ref{rigid_theorem}. Finally, in Section \ref{ineq}, we make note of some key inequalities that will be used throughout the remainder of this paper. 

\subsection{Slender body geometry}\label{geometry}
As in \cite{closed_loop}, we let $\X : \T\equiv \R / \Z \to \R^3$ denote the coordinates of a closed, non-self-intersecting $C^3$ curve in $\R^3$, parameterized by arclength $s$. We require that
\begin{equation}\label{cgamma}
 \inf_{s\neq s'}\frac{\abs{\X(s)-\X(s')}}{\abs{s-s'}} \ge c_\Gamma 
 \end{equation}
for some constant $c_\Gamma>0$. \\

Along $\X(s)$ we consider the orthonormal frame $(\be_t(s),\be_{n_1}(s),\be_{n_2}(s))$ defined in \cite{closed_loop}. Here $\be_t(s) = \frac{d\X}{ds}$ is the unit tangent vector to $\X(s)$ and $(\be_{n_1}(s),\be_{n_2}(s))$ span the plane normal to $\be_t(s)$. The frame satisfies the ODEs 
\begin{equation}\label{C2_frame}
\frac{d}{ds}\be_t = \kappa_1\be_{n_1}+\kappa_2\be_{n_2}, \quad \frac{d}{ds}\be_{n_1} = -\kappa_1\be_t+\kappa_3\be_{n_2}, \quad  \frac{d}{ds}\be_{n_2} = -\kappa_2\be_t-\kappa_3\be_{n_1}
\end{equation}
where $\kappa_1^2(s) +\kappa_2^2(s) = \kappa^2(s)$, the fiber curvature, and $\kappa_3$ is a constant satisfying $\abs{\kappa_3}\le \pi$. We require the orthonormal frame to be $C^2$ and denote
 \begin{equation}
 \kappa_{\max} := \max_{s\in\T} \abs{\kappa(s)}, \quad \xi_{\max} = \max_{s\in\T} \abs{\frac{\p^3\X}{\p s^3}} .
 \end{equation}
Note that $\abs{\p \kappa_1 /\p s} +\abs{\p \kappa_2/\p s} \le \xi_{\max}+2(\kappa_{\max}+\pi)$. \\

We define 
\[ \be_\rho(s,\theta) := \cos\theta\be_{n_1}(s)+\sin\theta\be_{n_2}(s) \]
and, for some $r_{\max}=r_{\max}(c_\Gamma,\kappa_{\max})\le \frac{1}{2\kappa_{\max}}$, we can uniquely parameterize points $\bx$ within a neighborhood ${\rm dist}(\bx,\X)< r_{\max}$ of $\X(s)$ as 
\[ \bx = \X(s) + \rho\be_\rho(s,\theta), \quad 0\le \rho <r_{\max}. \]

For $\epsilon<r_{\max}/4$, we may then define a slender body of uniform radius $\epsilon$ as
\begin{equation}\label{SB_def} 
\Sigma_\epsilon:= \big\{\bx\in \R^3 \; : \; \bx = \X(s) + \rho\be_\rho(s,\theta), \; \rho <\epsilon, \; 0\le \theta<2\pi \big\}.
\end{equation}
We parameterize the slender body surface $\Gamma_\epsilon = \p \Sigma_\epsilon$ as
\[ \Gamma_\epsilon =  \X(s) + \epsilon\be_\rho(s,\theta).\]
In addition, we may parameterize the Jacobian factor $\mc{J}_\epsilon(s,\theta)$ on the slender body surface as
\begin{equation}\label{jac_fac}
\mc{J}_\epsilon(s,\theta) = \epsilon\big(1-\epsilon(\kappa_1(s)\cos\theta+\kappa_2(s)\sin\theta) \big).
\end{equation}

\subsection{Variational form of \eqref{SB_PDE}}\label{variational}
Letting $\Omega_\epsilon=\R^3\backslash\overline{\Sigma_\epsilon}$ for $\Sigma_\epsilon$ as in Section \ref{geometry}, we recall the following function spaces, used in \cite{closed_loop} to study a slender body PDE of the form \eqref{SB_PDE_ps}. We use $D^{1,2}(\Omega_\epsilon)$ to denote the homogeneous Sobolev space
\begin{equation}\label{D12}
D^{1,2}(\Omega_\epsilon) = \big\{ \bu\in L^6(\Omega_\epsilon) \; : \; \nabla \bu \in L^2(\Omega_\epsilon) \big\}, 
\end{equation}
which, due to the Sobolev inequality in $\Omega_\epsilon \subset \R^3$ (see Lemma \ref{sobolev}), is a Hilbert space with norm $\norm{\nabla \bu}_{L^2(\Omega_\epsilon)}$. We define $D^{1,2}_0(\Omega_\epsilon)$ as the closure of $C_0^\infty(\Omega_\epsilon)$ (smooth, compactly supported test functions) in $D^{1,2}(\Omega_\epsilon)$.  \\

We also recall the space $\A_\epsilon$, the subspace of $D^{1,2}(\Omega_\epsilon)$ with $\theta$-independent boundary values: 
\begin{equation}\label{Aepsilon}
\A_\epsilon = \big\{ \bu\in D^{1,2}(\Omega_\epsilon) \; : \; \bu\big|_{\Gamma_\epsilon} = \bu(s) \big\}.
\end{equation}
Here the boundary value $\bu\big|_{\Gamma_\epsilon} = \bu(s)$ is not directly specified but is required to be independent of the surface angle $\theta$. We define $\A_\epsilon^\div$ to be the divergence-free subspace of $\A_\epsilon$. \\

We also recall the variational form of \eqref{SB_PDE_ps}, examined in detail in \cite{closed_loop}.
\begin{definition}[Weak solution to \eqref{SB_PDE_ps}]
A weak solution $\bu^{\rm p,s}\in \A_\epsilon^{\div}$ to \eqref{SB_PDE_ps} satisfies
\begin{equation}\label{weak_no_p}
\int_{\Omega_\epsilon} 2 \E(\bu^{\rm p,s}): \E(\bv) \, d\bx = \int_\T \bv(s)\cdot\bm{f}^{\rm s}(s) \, ds 
\end{equation}
for any $\bv\in\A_\epsilon^{\div}$. In addition, for $\bu^{\rm p,s}$ satisfying \eqref{weak_no_p}, there exists a unique pressure $p^{\rm p,s}\in L^2(\Omega_\epsilon)$ satisfying
\begin{equation}\label{weak_with_p}
\int_{\Omega_\epsilon} \big( 2 \E(\bu^{\rm p,s}): \E(\bv) - p^{\rm p,s}\, \div\bv \big)\, d\bx = \int_\T \bv(s)\cdot\bm{f}^{\rm s}(s) \, ds
\end{equation}
for any $\bv\in \A_\epsilon$.
\end{definition}

To study \eqref{SB_PDE}, we define the following subspace of $\A_\epsilon$, where we further restrict the boundary value to be a rigid motion:
\begin{equation}\label{Repsilon}
\mc{R}_\epsilon =  \big\{ \bu\in D^{1,2}(\Omega_\epsilon) \; : \; \bu\big|_{\Gamma_\epsilon} = \bv + \bm{\omega}\times\X(s) \text{ for }\bv,\, \bm{\omega}\in \R^3  \big\}.
\end{equation}
Again, $\bv$ and $\bm{\omega}$ are not directly specified but are required to be constant vectors in $\R^3$. We let $\mc{R}_\epsilon^\div$ denote the divergence-free subspace of $\mc{R}_\epsilon$. \\

We then define a weak solution to the rigid motion slender body PDE as follows.
\begin{definition}[Weak solution to \eqref{SB_PDE}]\label{rigid_weak}
A weak solution $\bu^{\rm p}\in \mc{R}_\epsilon^\div$ to \eqref{SB_PDE} satisfies
\begin{equation}
\int_{\Omega_\epsilon} 2\, \E(\bu^{\rm p}): \E(\bm{\varphi}) \, d\bx = \bv_\varphi \cdot\bm{F} + \bm{\omega}_\varphi \cdot\bm{T}
\end{equation}
for any $\bm{\varphi}\in \mc{R}_\epsilon^\div$, where we denote $\bm{\varphi}\big|_{\Gamma_\epsilon} =  \bv_\varphi + \bm{\omega}_\varphi \times \X(s)$. 
\end{definition}

Given the existence and uniqueness of $\bu^{\rm p}$ satisfying Definition \ref{rigid_weak}, using an essentially identical proof to that in Section 2.2 of \cite{closed_loop}, we can establish an equivalent notion of weak solution that includes a corresponding weak pressure $p^{\rm p}\in L^2(\Omega_\epsilon)$ and removes the divergence-free restriction on test functions $\bm{\varphi}$.
\begin{definition}[Weak solution to \eqref{SB_PDE} with pressure]\label{rigid_weak_p}
Given $\bu^{\rm p}\in \mc{R}_\epsilon^\div$ satisfying Definition \ref{rigid_weak}, there exists a unique $p^{\rm p}\in L^2(\Omega_\epsilon)$ satisfying
\begin{equation}
\int_{\Omega_\epsilon} \big(2\, \E(\bu^{\rm p}): \E(\bm{\varphi}) - p\, \div\bm{\varphi} \big) \, d\bx = \bv_\varphi \cdot\bm{F} + \bm{\omega}_\varphi\cdot \bm{T}
\end{equation}
for any $\bm{\varphi}\in \mc{R}_\epsilon$. Here we again denote $\bm{\varphi}\big|_{\Gamma_\epsilon} =  \bv_\varphi + \bm{\omega}_\varphi \times \X(s)$. 
\end{definition}

\subsection{Important inequalities}\label{ineq}
In addition to the definitions of Section \ref{variational}, we collect the statements of various inequalities that are used throughout the paper, keeping track of the $\epsilon$-dependence in any constants that arise. The proofs of these inequalities are mostly contained in \cite{closed_loop}, with the exception of Lemma \ref{trace2}, which appears in Appendix \ref{appendix}.  \\

First, we note the following pair of trace inequalities. The first holds for functions $\bu\in \A_\epsilon$ due to $\theta$-independence on $\Gamma_\epsilon$. As a slight abuse of notation, the trace operator ${\rm Tr}$, when applied to $\A_\epsilon$ functions, will be considered as both a function on $\Gamma_\epsilon$ and on $\T$. Note that for $\bu\in \A_\epsilon$, we have 
\begin{align*}
 \|{\rm Tr}(\bu)\|_{L^2(\Gamma_\epsilon)}^2 &= \int_\T \int_0^{2\pi} |{\rm Tr}(\bu)(s)|^2 \, \mc{J}_\epsilon(s,\theta) d\theta \,ds \\
 &=  \int_\T |{\rm Tr}(\bu)(s)|^2 \int_0^{2\pi} \mc{J}_\epsilon(s,\theta) d\theta \,ds = 2\pi \epsilon \|{\rm Tr}(\bu)\|_{L^2(\T)}^2, 
 \end{align*}
where we have used that $\int_0^{2\pi} \mc{J}_\epsilon(s,\theta) d\theta =2\pi\epsilon$ by \eqref{jac_fac}. For  $\bu\in \A_\epsilon$, the following lemma holds.
\begin{lemma}\emph{($L^2(\T)$ trace inequality)}\label{trace1} 
Let $\Omega_{\epsilon}=\R^3 \backslash \overline{\Sigma_{\epsilon}}$ be as in Section \ref{geometry}. Then any 
$\bu\in \A_\epsilon$ satisfies
\begin{equation}
\|{\rm Tr}(\bu)\|_{L^2(\T)} \le C |\log\epsilon|^{1/2} \| \nabla \bu\|_{L^2(\Omega_{\epsilon})}, 
\end{equation}
where the constant $C$ depends on $\kappa_{\max}$ and $c_{\Gamma}$ but is independent of $\epsilon$. 
\end{lemma}
The proof of this lemma appears in Appendix A.2.1 of \cite{closed_loop}. \\

On the other hand, for general $D^{1,2}(\Omega)$ functions, the following trace inequality holds over the surface $\Gamma_\epsilon$:
\begin{lemma}\emph{($L^2(\Gamma_\epsilon)$ trace inequality)}\label{trace2} 
Let $\Omega_{\epsilon}=\R^3 \backslash \overline{\Sigma_{\epsilon}}$ be as in Section \ref{geometry}. Then any $\bu\in D^{1,2}(\Omega_\epsilon)$ satisfies
\begin{equation}
\|{\rm Tr}(\bu)\|_{L^2(\Gamma_\epsilon)} \le C \sqrt{\epsilon}|\log\epsilon|^{1/2} \| \nabla \bu\|_{L^2(\Omega_{\epsilon})}, 
\end{equation}
where the constant $C$ depends on $\kappa_{\max}$ and $c_{\Gamma}$ but is independent of $\epsilon$. 
\end{lemma}
The proof of Lemma \ref{trace2} appears in Appendix \ref{appendix}. \\

We will also need the following Korn inequality.
\begin{lemma}\emph{(Korn inequality)}\label{korn}
Let $\Omega_{\epsilon}=\R^3 \backslash \overline{\Sigma_{\epsilon}}$ be as in Section \ref{geometry}. Then any $\bu\in D^{1,2}(\Omega_{\epsilon})$ satisfies 
\begin{equation}
 \|\nabla \bu\|_{L^2(\Omega_{\epsilon})} \le C\|\E(\bu)\|_{L^2(\Omega_{\epsilon})}, 
 \end{equation}
 where the constant $C$ depends only on $\kappa_{\max}$ and $c_{\Gamma}$.
\end{lemma}
The proof of $\epsilon$-independence in the Korn constant is given in Appendix A.2.2 - A.2.3 in \cite{closed_loop}. \\

Finally, we make use of the following pressure estimate.
\begin{lemma}\label{pressure}
For $(\bu,p)$ satisfying the Stokes equations in $\Omega_\epsilon$, we have
\begin{equation}
\norm{p}_{L^2(\Omega_\epsilon)} \le C \norm{\E(\bu)}_{L^2(\Omega_\epsilon)}
\end{equation}
for $C$ independent of $\epsilon$.
\end{lemma}
The proof of this lemma exactly follows the proof of estimate (2.17) in \cite{closed_loop}. 

\section{Proof of Lemma \ref{v_omega_bd} and a corollary}\label{2ndLemma} 
Here we prove Lemma \ref{v_omega_bd} and make note of a corollary which allows us to obtain a useful bound for functions in $\mc{R}_\epsilon$. This corollary, along with the Korn inequality (Lemma \ref{korn}) and pressure estimate (Lemma \ref{pressure}), then allows us to prove Theorem \ref{SB_PDE_well}.

\begin{proof}[Proof of Lemma \ref{v_omega_bd}:]
Note that Lemma \ref{v_omega_bd} is obviously true when $\bv=\bm{\omega}=0$; thus we can assume that at least one of $\bv,\bm{\omega}$ is nonzero. Suppose that Lemma \ref{v_omega_bd} does not hold. Then we may choose a sequence of triples $(\bv_k,\bm{\omega}_k,\X_k(s))$ such that the following properties hold for each $k=1,2,3,\dots$. First, $\bv_k,\bm{\omega}_k\in\R^3$ satisfy $\abs{\bv_k}^2+\abs{\bm{\omega}_k}^2=1$, and $\X_k(s)$ is a closed curve satisfying the geometric constraints of Section \ref{geometry} -- in particular, $\abs{\X_k''}\le \kappa_{\max}$. In addition,
 \[ 1 = \abs{\bv_k}^2+\abs{\bm{\omega}_k}^2 > k^2\int_{\T}\abs{\bv_k+\bm{\omega}_k\times \X_k(s)}^2 \, ds.\]
 Then 
 \[ \int_{\T}\abs{\bv_k+\bm{\omega}_k\times \X_k(s)}^2 \, ds < \frac{1}{k^2} \to 0\]
 as $k\to\infty$. Since $\bv_k,\bm{\omega}_k$ are just vectors in $\R^3$, some limit $\bv_\infty,\bm{\omega}_\infty$ exists. Furthermore, since each $\X_k$ is controlled in $C^2$ by $\kappa_{\max}$, we have that (passing to a subsequence) $\X_k\to \X_\infty$ in $C^1$ for some closed, unit length curve $\X_\infty(s)$. Thus
 \[ \int_{\T}\abs{\bv_\infty +\bm{\omega}_\infty \times \X_\infty(s)}^2 \, ds =0, \]
 and therefore $\bm{\omega}_\infty\times \X_\infty(s)\equiv -\bv_\infty$. But $\bm{\omega}_\infty$ and $\bv_\infty$ are both constant vectors with $\abs{\bv_\infty}^2+\abs{\bm{\omega}_\infty}^2=1$, while $\X_\infty(s)$ necessarily has nonzero curvature. Thus $\bm{\omega}_\infty\times \X_\infty(s)$ cannot identically equal the constant vector $-\bv_\infty$. Furthermore, because $\X_k$ was allowed to vary among curves satisfying the constraints of Section \ref{geometry}, the constant $C$ arising in Lemma \ref{v_omega_bd} depends only on $c_\Gamma$ and $\kappa_{\max}$. 
\end{proof}

Given $\bm{F}$ and $\bm{T}$, as an immediate corollary to Lemma \ref{v_omega_bd} we obtain the following useful bound for any function $\bm{\varphi}\in\mc{R}_\epsilon$. 
\begin{corollary}\label{omegaP}
Consider $\bm{\varphi}\in \mc{R}_\epsilon$ with boundary value denoted by $\bv_\varphi+\bm{\omega}_\varphi\times \X(s)$. Then 
\begin{equation}\label{curved_est}
|\bv_\varphi|+|\bm{\omega}_\varphi| \le C |\log\epsilon|^{1/2} \norm{\E(\bm{\varphi})}_{L^2(\Omega_\epsilon)}
\end{equation}
where $C$ depends only on $c_\Gamma$ and $\kappa_{\max}$. 
\end{corollary}

\begin{proof}
Using Lemma \ref{v_omega_bd} along with the slender body trace estimate (Lemma \ref{trace1}) and Korn inequality (Lemma \ref{korn}), we immediately obtain
\begin{align*} 
\abs{\bv_\varphi}+ \abs{\bm{\omega}_\varphi} &\le \norm{\bv_\varphi+\bm{\omega}_\varphi\times\X}_{L^2(\T)} \le C\abs{\log\epsilon}^{1/2}\norm{\nabla \bm{\varphi}}_{L^2(\Omega_\epsilon)} \\
&\le C\abs{\log\epsilon}^{1/2}\norm{\E(\bm{\varphi})}_{L^2(\Omega_\epsilon)}
\end{align*}
for $C$ depending only on $c_\Gamma$ and $\kappa_{\max}$. 
\end{proof}

Using Corollary \ref{omegaP} and the variational formulation of \eqref{SB_PDE}, we may now prove Theorem \ref{SB_PDE_well}. 
\begin{proof}[Proof of Theorem \ref{SB_PDE_well}]
We first show the existence of a weak solution $\bu^{\rm p}\in \mc{R}_\epsilon^{\div}$ satisfying Definition \ref{rigid_weak}. Note that the bilinear form appearing on the left hand side of Definition \ref{rigid_weak} is bounded on $\mc{R}_\epsilon^{\div}$, as
\[ \abs{\int_{\Omega_\epsilon} 2\, \E(\bu^{\rm p}): \E(\bm{\varphi}) \, d\bx} \le 2\norm{ \E(\bu^{\rm p})}_{L^2(\Omega_\epsilon)}\norm{\E(\bm{\varphi})}_{L^2(\Omega_\epsilon)} \le 2\norm{ \nabla\bu^{\rm p} }_{L^2(\Omega_\epsilon)}\norm{\nabla\bm{\varphi}}_{L^2(\Omega_\epsilon)}.\]
Coercivity of the bilinear form also follows by the Korn inequality (Lemma \ref{korn}). Furthermore, using Corollary \ref{omegaP}, the linear functional on the right hand side of Definition \ref{rigid_weak} is bounded for $\bm{\varphi}\in \mc{R}_\epsilon^{\div}$, as 
\[ \abs{ \bv_\varphi \cdot\bm{F} + \bm{\omega}_\varphi\cdot \bm{T}} \le C\abs{\log\epsilon}^{1/2}\norm{\E(\bm{\varphi})}_{L^2(\Omega_\epsilon)}(\abs{\bm{F}} +\abs{\bm{T}}) \le C\abs{\log\epsilon}^{1/2}\norm{\nabla\bm{\varphi}}_{L^2(\Omega_\epsilon)}(\abs{\bm{F}} +\abs{\bm{T}}). \]
Then, by the Lax-Milgram theorem, there exists a unique weak solution $\bu^{\rm p}\in \mc{R}_\epsilon^{\div}$ to \eqref{SB_PDE}. \\

In addition, using the variational form of \eqref{SB_PDE} along with Corollary \ref{omegaP}, we have that $\bu^{\rm p}$ satisfies
\begin{align*}
\int_{\Omega_\epsilon} \abs{\E(\bu^{\rm p})}^2 \, d\bx &= \int_{\Gamma_\epsilon} (\bv^{\rm p}+\bm{\omega}^{\rm p}\times \X(s))\cdot(\bm{\sigma}^{\rm p}\bm{n}) \, dS \\
&= \bv^{\rm p}\cdot\int_{\Gamma_\epsilon} \bm{\sigma}^{\rm p}\bm{n} \, dS + \bm{\omega}^{\rm p}\cdot \int_{\T} \X(s)\times \bigg(\int_0^{2\pi} (\bm{\sigma}^{\rm p}\bm{n}) \, \mc{J}_\epsilon(s,\theta) d \theta \bigg) ds \\
&\le \abs{\bv^{\rm p}}\abs{\bm{F}} + \abs{\bm{\omega}^{\rm p}}\abs{\bm{T}} \le C\abs{\log\epsilon}^{1/2}\norm{\E(\bu^{\rm p})}_{L^2(\Omega_\epsilon)}(\abs{\bm{F}} +\abs{\bm{T}}) \\
&\le \frac{1}{2}\norm{\E(\bu^{\rm p})}_{L^2(\Omega_\epsilon)}^2 + C\abs{\log\epsilon}(\abs{\bm{F}}^2 +\abs{\bm{T}}^2),
\end{align*}
where we have used Young's inequality in the last line. We thus obtain the estimate 
\begin{equation}\label{Eup_est}
\norm{\E(\bu^{\rm p})}_{L^2(\Omega_\epsilon)} \le C\abs{\log\epsilon}^{1/2} ( \abs{\bm{F}}+\abs{\bm{T}} ).
\end{equation}

As noted after Definition \ref{rigid_weak}, the existence of a unique corresponding weak pressure $p^{\rm p}\in L^2(\Omega_\epsilon)$ satisfying Definition \ref{rigid_weak_p} as well as Lemma \ref{pressure} follows by an essentially identical proof to that appearing in Section 2.2 of \cite{closed_loop}. \\

Combining \eqref{Eup_est} with Lemmas \ref{korn} and \ref{pressure} then yields the bound \eqref{totalPest}.
\end{proof}

\section{Classical versus slender body PDE description of rigid motion}\label{3rdLemma}
Using the variational framework of Section \ref{variational0} along with Lemma \ref{v_omega_bd}, we prove Lemma \ref{true_vs_SB} comparing the classical PDE \eqref{rigid} and slender body PDE \eqref{SB_PDE} descriptions of rigid slender body motion.

\begin{proof}[Proof of Lemma \ref{true_vs_SB}:]
The difference $\overline\bu= \bu^{\rm r}- \bu^{\rm p}$, $\overline p= p^{\rm r}-  p^{\rm p}$, $\overline{\bm{\sigma}}=\bm{\sigma}^{\rm r}- \bm{\sigma}^{\rm p}$, $\overline{\bm{\omega}}= \bm{\omega}^{\rm r}- \bm{\omega}^{\rm p}$, $\overline\bv= \bv^{\rm r}- \bv^{\rm p}$ satisfies the PDE
\begin{equation}\label{diff_PDE}
\begin{aligned}
-\Delta \overline\bu +\nabla \overline p &=0, \quad \div \overline\bu =0 \hspace{2.3cm} \text{in } \Omega_\epsilon \\
\overline\bu(\bx) &= \overline\bv+ \overline{\bm{\omega}}\times \bx + \epsilon \bm{\omega}^{\rm p} \times \be_\rho, \qquad \bx \in \Gamma_\epsilon \\
\overline\bu(\bx) &\to 0 \hspace{4.3cm} \text{as }\abs{\bx}\to \infty \\
\int_{\Gamma_\epsilon} \overline{\bm{\sigma}}\bm{n} \, dS &=0, \quad \int_{\Gamma_\epsilon} \bx \times (\overline{\bm{\sigma}}\bm{n}) \, dS = - \epsilon \int_{\Gamma_\epsilon} \be_\rho \times (\bm{\sigma}^{\rm p}\bm{n}) \; dS.
\end{aligned}
\end{equation}

Then, multiplying \eqref{diff_PDE} by $\overline\bu$ and integrating by parts, we have that $\overline\bu$ satisfies 
\begin{equation}\label{diff_bound0}
\begin{aligned}
\int_{\Omega_\epsilon} 2|\E(\overline\bu)|^2 \, d\bx &= \int_{\Gamma_\epsilon} \big( \overline\bv+ \overline{\bm{\omega}}\times \bx + \epsilon\bm{\omega}^{\rm p}\times\be_\rho \big) \cdot(\overline{\bm{\sigma}}\bm{n}) \, dS \\
&= \overline\bv\cdot\int_{\Gamma_\epsilon}\overline{\bm{\sigma}}\bm{n} \, dS + \overline{\bm{\omega}}\cdot\int_{\Gamma_\epsilon} \bx\times(\overline{\bm{\sigma}}\bm{n}) \, dS + \epsilon\bm{\omega}^{\rm p}\cdot\int_{\Gamma_\epsilon} \be_\rho\times(\overline{\bm{\sigma}}\bm{n}) \, dS \\
&= -\epsilon \overline{\bm{\omega}}\cdot\int_{\Gamma_\epsilon} \be_\rho \times (\bm{\sigma}^{\rm p}\bm{n}) \, dS + \epsilon \bm{\omega}^{\rm p}\cdot\int_{\Gamma_\epsilon} \be_\rho \times(\overline{\bm{\sigma}}\bm{n}) \, dS.
\end{aligned}
\end{equation}

To estimate the right hand side of \eqref{diff_bound0}, we first need to define a smooth cutoff function $\phi(\rho)$ satisfying
\begin{equation}\label{cutoff_def}
\phi(\rho) = \begin{cases}
1, & \rho < 2 \\
0, & \rho >4
\end{cases}
\end{equation}
with smooth decay satisfying 
\begin{equation}\label{phi_decay}
\abs{\frac{d\phi}{d\rho}} \le c_\phi.
\end{equation}
Then for $\bx=\X(s)+\rho\be_\rho(\theta,s)$ in a neighborhood of $\Gamma_\epsilon$, we define $\phi_\epsilon(\rho):= \phi(\rho/\epsilon)$.\\
 
We estimate the second term on the right hand side first, noting that the estimation of the first term will be essentially identical. Using index notation (the subscript $\cdot_{,j}$ signifies $\frac{\p \cdot}{\p x_j}$; sum over repeated indices) along with the divergence theorem, we may write 
\begin{equation}\label{cross_est0}
\begin{aligned} 
\int_{\Gamma_\epsilon} \big(\be_\rho \times(\overline{\bm{\sigma}}\bm{n}) \big)_i \, dS &=  \int_{\Gamma_\epsilon} \varepsilon_{ijk} (e_\rho)_j \overline\sigma_{k\ell} n_\ell \, dS = \int_{\Omega_\epsilon} (\phi \, \varepsilon_{ijk}(e_\rho)_j \overline\sigma_{k\ell})_{,\ell} \, d\bx \\
&= \int_{\Omega_\epsilon} \varepsilon_{ijk}\big( \phi_{,\ell}(e_\rho)_j \overline\sigma_{k\ell} + \phi \,(e_\rho)_{j,\ell} \overline\sigma_{k\ell} \big)\, d\bx.
\end{aligned}
\end{equation}
Here $\varepsilon_{ijk}$ is the alternating symbol
\[ \varepsilon_{ijk} = \begin{cases}
1, & \text{ for even permutations of }i,j,k \\
-1, & \text{ for odd permutations of }i,j,k \\
0, & \text{ if } i=j,j=k, \text{ or }k=i, \end{cases} \]
and we have used that $\overline{\bm{\sigma}}$ is divergence-free. 


Now, due to the cutoff $\phi_\epsilon$, the integrand on the right hand side of \eqref{cross_est0} is supported only within the region 
\[ \mc{O}_\epsilon := \big\{\X(s)+\rho \be_\rho(s,\theta) \; : \; s\in\T, \; \epsilon \le \rho\le4\epsilon, \; 0\le \theta<2\pi \big\} \]
with $\abs{\mc{O}_\epsilon}= C\epsilon^2$ for some $C$ depending only on $c_\Gamma$ and $\kappa_{\max}$. \\

Within $\mc{O}_\epsilon$, defining $\wh\kappa(s,\theta):= \kappa_1(s)\cos\theta+\kappa_2(s)\sin\theta$, we have 
\begin{equation}\label{grad_rho}
\begin{aligned}
\abs{\nabla \be_\rho(s,\theta)} &=  \abs{\frac{1}{\rho}\frac{\p \be_\rho}{\p\theta}\be_\theta^{\rm T}+ \frac{1}{1-\rho\wh\kappa}\bigg(\frac{\p\be_\rho}{\p s} - \kappa_3\frac{\p \be_\rho}{\p\theta} \bigg)\be_t^{\rm T} } \\
&= \abs{\frac{1}{\rho}\be_\theta\be_\theta^{\rm T} - \frac{1}{1-\rho\wh\kappa}(\kappa_1\cos\theta+\kappa_2\sin\theta) \be_t\be_t^{\rm T} } \le \frac{1}{\epsilon}+ 4\kappa_{\max},
 \end{aligned}
 \end{equation}
 where the final $\kappa_{\max}$ bound is shown in Appendix \ref{reg_lem}.  \\

Using \eqref{phi_decay}, \eqref{grad_rho}, and Cauchy-Schwarz, we may estimate \eqref{cross_est0} as
\begin{equation}\label{cross_bound}
\begin{aligned}
\abs{\int_{\Gamma_\epsilon} \be_\rho \times(\overline{\bm{\sigma}}\bm{n}) \, dS} &\le \int_{\mc{O}_\epsilon} \big(\abs{\phi_\epsilon\nabla\be_\rho}+\abs{\nabla\phi_\epsilon})\abs{\overline{\bm{\sigma}}} \, d\bx \\
&\le \abs{\mc{O}_\epsilon}^{1/2} \bigg(\frac{1}{\epsilon} + 4\kappa_{\max} + \frac{c_\phi}{\epsilon} \bigg)\bigg(\int_{\Omega_\epsilon} \abs{\overline{\bm{\sigma}}}^2 \, d\bx\bigg)^{1/2} \\
& \le C \bigg(\int_{\Omega_\epsilon} \abs{\overline{\bm{\sigma}}}^2 \, d\bx\bigg)^{1/2}
\end{aligned}
\end{equation}
where $C$ depends only on the shape of $\X$ -- in particular, $c_\Gamma$ and $\kappa_{\max}$. Finally, using Lemma \ref{pressure}, we obtain
\begin{equation}\label{cross_bound2}
\abs{\int_{\Gamma_\epsilon} \be_\rho \times(\overline{\bm{\sigma}}\bm{n}) \, dS} \le C \bigg(\int_{\Omega_\epsilon} \big(\abs{\E(\overline\bu)}^2 + \overline p^2 \big) \, d\bx\bigg)^{1/2} \le C \bigg(\int_{\Omega_\epsilon} \abs{\E(\overline\bu)}^2 \, d\bx\bigg)^{1/2}.
\end{equation}

Following exactly the same procedure, we can also show 
\begin{equation}\label{cross_bound3}
\abs{\int_{\Gamma_\epsilon} \be_\rho \times(\bm{\sigma}^{\rm p}\bm{n}) \, dS} \le C \bigg(\int_{\Omega_\epsilon} \abs{\E(\bu^{\rm p})}^2 \, d\bx\bigg)^{1/2}.
\end{equation}

Furthermore, in the same way as in Lemma \ref{omegaP}, it can be shown that 
\begin{equation}\label{omegav_bound}
\abs{\overline{\bm{\omega}}}+ \abs{\overline\bv} \le C\norm{\overline\bv+\overline{\bm{\omega}}\times\bx+ \epsilon\bm{\omega}^{\rm p}\times \be_\rho}_{L^2(\Gamma_\epsilon)} \le C\sqrt{\epsilon}\abs{\log\epsilon}^{1/2}\norm{\E(\overline\bu)}_{L^2(\Omega_\epsilon)}.
\end{equation}
Note that the first inequality holds via a similar contradiction as in the proof of Lemma \ref{v_omega_bd}, except here we must use that $\overline\bv+\overline{\bm{\omega}}\times\bx+ \epsilon\bm{\omega}^{\rm p}\times \be_\rho=\overline\bv+\overline{\bm{\omega}}\times\X(s)+ \epsilon\bm{\omega}^{\rm r}\times \be_\rho(s,\theta)$ for $\bx\in \Gamma_\epsilon$. The analogous contradiction arises from the fact that $\epsilon\bm{\omega}^{\rm r}\times \be_\rho(s,\theta)$ depends on $\theta$, whereas $\overline\bv+\overline{\bm{\omega}}\times\X(s)$) does not. For the second inequality we have used the $L^2(\Gamma_\epsilon)$ trace estimate (Lemma \ref{trace2}) and the Korn inequality (Lemma \ref{korn}).\\

Then, using \eqref{cross_bound2} and \eqref{cross_bound3} in \eqref{diff_bound0} along with Lemma \ref{omegaP} and \eqref{omegav_bound}, we have
\begin{equation}\label{diff_bound1}
\begin{aligned}
\int_{\Omega_\epsilon} 2|\E(\overline\bu)|^2 \, d\bx &\le \epsilon C\abs{\overline{\bm{\omega}}}\bigg(\int_{\Omega_\epsilon}\abs{\E(\bu^{\rm p})}^2 \, d\bx \bigg)^{1/2} + \epsilon C\abs{\bm{\omega}^{\rm p}} \bigg(\int_{\Omega_\epsilon}\abs{\E(\overline\bu)}^2 \, d\bx \bigg)^{1/2} \\
&\le \epsilon\abs{\log\epsilon}^{1/2}C\bigg(\int_{\Omega_\epsilon}\abs{\E(\bu^{\rm p})}^2 \, d\bx \bigg)^{1/2}\bigg(\int_{\Omega_\epsilon}\abs{\E(\overline\bu)}^2 \, d\bx \bigg)^{1/2} \\
&\le \epsilon^2\abs{\log\epsilon}C\int_{\Omega_\epsilon}\abs{\E(\bu^{\rm p})}^2 \, d\bx + \int_{\Omega_\epsilon}\abs{\E(\overline\bu)}^2 \, d\bx,
\end{aligned}
\end{equation}
where we have used Young's inequality in the last line. Then, using \eqref{Eup_est}, we obtain
\begin{equation}
\norm{\E(\overline\bu)}_{L^2(\Omega_\epsilon)} \le \epsilon\abs{\log\epsilon}C( \abs{\bm{T}}+ \abs{\bm{F}}).
\end{equation}

Finally, using \eqref{omegav_bound} again, we obtain Lemma \ref{true_vs_SB}.
\end{proof}

\section{Proof of Lemma \ref{ups_up_err} }\label{1stLemma}
Finally, we prove Lemma \ref{ups_up_err} comparing the rigid slender body PDE \eqref{SB_PDE} to the intermediary slender body PDE \eqref{SB_PDE_ps}. \\

We begin by defining 
\begin{equation}\label{fp_def}
\bm{f}^{\rm p}(s):=\int_0^{2\pi} (\bm{\sigma}^{\rm p}\bm{n}) \, \mc{J}_\epsilon(s,\theta) d\theta
\end{equation}
for $\bm{\sigma}^{\rm p}$ as in \eqref{SB_PDE}, and establish the following: 
\begin{lemma}\label{fp_FT}
Suppose the slender body $\Sigma_\epsilon$ is as in Section \ref{geometry} -- in particular, $\X\in C^3(\T)$. Let the total force $\bm{F}$ and torque $\bm{T}$ be given, and let $\bm{f}^{\rm p}$ be as defined in \eqref{fp_def}. Then 
\begin{equation}
 \norm{\bm{f}^{\rm p}}_{L^2(\T)} \le C\abs{\log\epsilon}^{3/2}(\abs{\bm{F}}+\abs{\bm{T}})
 \end{equation}
 for $C$ depending only on $c_\Gamma$, $\kappa_{\max}$, and $\xi_{\max}$.  
 \end{lemma}

\begin{proof}
The proof of this lemma relies on a higher regularity estimate for $\bm{\sigma}^{\rm p}$. Note that once Theorem \ref{SB_PDE_well} has been established, we immediately obtain that $\bu^{\rm p}\big|_{\Gamma_\epsilon} = \bv^{\rm p}+\bm{\omega}^{\rm p}\times \X(s)$ is in $C^3(\Omega_\epsilon)$, since $\bv^{\rm p}$ and $\bm{\omega}^{\rm p}$ are just constants in $\R^3$ and the fiber centerline $\X$ is in $C^3(\T)$. Given this $C^3$ Dirichlet data, $\bm{\sigma}^{\rm p}\in H^1(\Omega_\epsilon)$ follows by standard higher regularity arguments for the exterior Stokes Dirichlet boundary value problem (see the proof of Lemma V.4.3 in \cite{galdi2011introduction} or Theorem IV.5.8 in \cite{boyer2012mathematical}). Note that since $\X\in C^3(\T)$, $\bm{\sigma}^{\rm p}$ should in fact be even more regular, but the method we use to show Lemma \ref{high_reg} only allows us to quantify the $\epsilon$-dependence in the estimate for $\norm{\nabla\bm{\sigma}^{\rm p}}_{L^2(\Omega_\epsilon)}$. In particular, we can show the following bound on $\nabla \bm{\sigma}^{\rm p}$. 
\begin{lemma}\label{high_reg}
Given $\Omega_\epsilon$ as in Section \ref{geometry}, the solution $\bm{\sigma}^{\rm p}$ to \eqref{SB_PDE} belongs to $H^1(\Omega_\epsilon)$ and satisfies 
\begin{equation}\label{high_reg_eqn}
\norm{\nabla\bm{\sigma}^{\rm p}}_{L^2(\Omega_\epsilon)} \le \norm{\nabla^2\bu^{\rm p}}_{L^2(\Omega_\epsilon)} + \norm{\nabla p^{\rm p}}_{L^2(\Omega_\epsilon)} \le \frac{C}{\epsilon}\abs{\log\epsilon}^{1/2}\big(\norm{\nabla\bu^{\rm p}}_{L^2(\Omega_\epsilon)} + \norm{p^{\rm p}}_{L^2(\Omega_\epsilon)} \big), 
\end{equation}
where $C$ depends on $c_\Gamma$, $\kappa_{\max}$, and $\xi_{\max}$.
\end{lemma}
The proof of the $\epsilon$-dependence in Lemma \ref{high_reg} is given in Appendix \ref{reg_lem}. \\

Using Lemma \ref{high_reg} and Corollary \ref{omegaP}, we have the higher regularity estimate 
\begin{equation}\label{high_reg2}
\begin{aligned}
\norm{\nabla\bm{\sigma}^{\rm p}}_{L^2(\Omega_\epsilon)} &\le \frac{C}{\epsilon}\abs{\log\epsilon}^{1/2}\bigg(\norm{\nabla\bu^{\rm p}}_{L^2(\Omega_\epsilon)} + \norm{p^{\rm p}}_{L^2(\Omega_\epsilon)} \bigg) \le \frac{C}{\epsilon}\abs{\log\epsilon}(\abs{\bm{F}}+\abs{\bm{T}}).
\end{aligned}
\end{equation}

Now, using that $\mc{J}_\epsilon(s,\theta)>0$ for each $(s,\theta)\in \Gamma_\epsilon$ and the surface measure $\abs{\Gamma_\epsilon} = \int_\T \int_0^{2\pi} \mc{J}_\epsilon(s,\theta) d\theta ds =  \epsilon$, we have 
\begin{align*}
\norm{\bm{f}^{\rm p}}_{L^2(\T)}^2 &= \int_\T \abs{\int_0^{2\pi} \bm{\sigma}^{\rm p}\bm{n} \, \mc{J}_\epsilon(s,\theta) d\theta }^2 ds \le \abs{\Gamma_\epsilon} \int_\T \int_0^{2\pi} \abs{{\rm Tr}(\bm{\sigma}^{\rm p})}^2 \, \mc{J}_\epsilon(s,\theta) d\theta \, ds \\
&\le C\epsilon^2\abs{\log\epsilon}\norm{\nabla\bm{\sigma}^{\rm p}}_{L^2(\Omega_\epsilon)}^2 \le C\abs{\log\epsilon}^3(\abs{\bm{F}}+\abs{\bm{T}})^2.
\end{align*} 
Here we have applied both the $L^2(\Gamma_\epsilon)$ trace inequality (Lemma \ref{trace2}) and the higher regularity estimate \eqref{high_reg2} in the last line. 
\end{proof}

With Lemma \ref{fp_FT}, we are now equipped to show Lemma \ref{ups_up_err}.

\begin{proof}[Proof of Lemma \ref{ups_up_err}:]
The proof relies on estimates for the PDE satisfied by the difference between solutions to \eqref{SB_PDE} and \eqref{SB_PDE_ps}. Letting $\wt\bu = \bu^{\rm p,s}-\bu^{\rm p}$, $\wt p = p^{\rm p,s}- p^{\rm p} $, $\wt\bv = \bv^{\rm s}-\bv^{\rm p}$, $\wt{\bm{\omega}} = \bm{\omega}^{\rm s}-\bm{\omega}^{\rm p}$, $\wt{\bm{\sigma}}=\bm{\sigma}^{\rm p,s} - \bm{\sigma}^{\rm p}$, we consider the following boundary value problem:
\begin{equation}\label{wtbu_eqn}
\begin{aligned}
-\Delta \wt\bu +\nabla \wt p &=0, \quad \div \wt \bu =0 \hspace{2.25cm} \text{in } \Omega_\epsilon \\
\wt\bu(\bx) &= \wt\bv + \wt{\bm{\omega}}\times \X(s) + \bm{R}(s), \qquad \bx \in \Gamma_\epsilon \\
\wt\bu(\bx) &\to 0 \hspace{4.3cm} \text{as }\abs{\bx}\to \infty \\
\int_{\Gamma_\epsilon} \wt{\bm{\sigma}}\bm{n} \; dS &= 0, \quad \int_{\T} \X(s)\times \bigg(\int_0^{2\pi}  \wt{\bm{\sigma}}\bm{n} \, \mc{J}_\epsilon(s,\theta) d\theta \bigg) ds = 0,
\end{aligned}
\end{equation}
where $\bm{R}(s):={\rm Tr}(\bu^{\rm p,s})(s) - \big( \bv^{\rm s} + \bm{\omega}^{\rm s}\times \X(s)\big)$ satisfies
\begin{equation}\label{Rest}
\norm{\bm{R}}_{L^2(\T)} \le C\epsilon\abs{\log\epsilon}^{3/2}\norm{\bm{f}^{\rm s}}_{C^1(\T)} + \big\| r_\epsilon[\bm{f}^{\rm s}] \big\|_{L^2(\T)},
\end{equation}
by \eqref{centerline_err}. We consider the variational form of \eqref{wtbu_eqn}: multiplying by \eqref{wtbu_eqn} by $\wt\bu$ and integrating by parts, we have 
\begin{align*}
\int_{\Omega_\epsilon} 2|\E(\wt\bu)|^2 \, d\bx &= \int_{\Gamma_\epsilon} \big(\wt\bv + \wt{\bm{\omega}}\times \X(s) +\bm{R}(s) \big) \cdot(\wt{\bm{\sigma}}\bm{n}) \; dS \\
&= \wt\bv \cdot \int_{\Gamma_\epsilon}(\wt{\bm{\sigma}}\bm{n}) \; dS +  \wt{\bm{\omega}}\cdot\int_{\T} \X(s)\times \bigg(\int_0^{2\pi} (\wt{\bm{\sigma}}\bm{n})\, \mc{J}_\epsilon(s,\theta) d\theta \bigg) ds \\
&\qquad + \int_{\T} \bm{R}(s) \cdot \bigg(\int_0^{2\pi} (\wt{\bm{\sigma}}\bm{n}) \, \mc{J}_\epsilon(s,\theta) d\theta \bigg) ds \\
&=\int_\T \bm{R}(s) \cdot \big( \bm{f}^{\rm s} -\bm{f}^{\rm p} \big) \; ds \le \norm{\bm{R}}_{L^2(\T)} \big( \norm{\bm{f}^{\rm s}}_{L^2(\T)} + \norm{\bm{f}^{\rm p}}_{L^2(\T)} \big) \\
&\le C\big(\epsilon\abs{\log\epsilon}^{3/2}\norm{\bm{f}^{\rm s}}_{C^1(\T)} + \big\| r_\epsilon[\bm{f}^{\rm s}] \big\|_{L^2(\T)}\big)\big(\norm{\bm{f}^{\rm s}}_{L^2(\T)} + \abs{\log\epsilon}^{3/2}(\abs{\bm{F}}+\abs{\bm{T}})  \big) \\
& \le C\big( \epsilon\abs{\log\epsilon}^3\norm{\bm{f}^{\rm s}}_{C^1(\T)}^2+  \epsilon\abs{\log\epsilon}^3(\abs{\bm{F}}+\abs{\bm{T}})^2 + \epsilon^{-1}\big\| r_\epsilon[\bm{f}^{\rm s}] \big\|_{L^2(\T)}^2 \big).
\end{align*}
Here we have used \eqref{Rest} and Lemma \ref{fp_FT} in the second-to-last line.

\begin{remark}
It would seem to make sense to try to bound the difference $\bm{f}^{\rm s} -\bm{f}^{\rm p}$ appearing in the third equality by $\norm{\E(\wt\bu)}_{L^2(\Omega_\epsilon)}$, or try to use an extension $\overline{\bm{R}}(\bx)\in D^{1,2}(\Omega_\epsilon)$ with $\overline{\bm{R}}\big|_{\Gamma_\epsilon} = \bm{R}(s)$ and instead take $\wt\bu - \overline{\bm{R}}$ as a test function in the above variational estimate to get rid of the boundary term. In either case, we run into difficulties in that we only have an $L^2(\T)$ estimate for $\bm{R}(s)$, when at least an $H^{1/2}(\T)$ estimate would be needed. However, as noted in Lemma \ref{high_reg}, bounding the gradient of a function on $\Omega_\epsilon$ incurs an additional factor of $1/\epsilon$. By scaling, an $H^{1/2}(\T)$ estimate for $\bm{R}(s)$ would likely yield the same $\sqrt{\epsilon}$ factor appearing in Lemma \ref{ups_up_err}.   
\end{remark}

Now, using the $L^2(\T)$ trace inequality (Lemma \ref{trace1}), the Korn inequality (Lemma \ref{korn}), and Young's inequality, along the with above $\norm{\E(\wt\bu)}_{L^2(\Omega_\epsilon)}$ estimate, we have 
\begin{align*}
\|{\rm Tr}(\bu^{\rm p,s}) &- (\bv^{\rm p}+\bm{\omega}^{\rm p}\times\X)\|_{L^2(\T)} \le C\abs{\log\epsilon}^{1/2}\norm{\nabla\wt\bu}_{L^2(\Omega_\epsilon)} \le C\abs{\log\epsilon}^{1/2}\norm{\E(\wt\bu)}_{L^2(\Omega_\epsilon)} \\
&\le \sqrt{\epsilon}\abs{\log\epsilon}^{3/2}C\big(\norm{\bm{f}^{\rm s}}_{C^1(\T)}+ \abs{\bm{F}}+\abs{\bm{T}}\big) + C\epsilon^{-1/2}\big\| r_\epsilon[\bm{f}^{\rm s}] \big\|_{L^2(\T)} ,
\end{align*}
yielding Lemma \ref{ups_up_err}. 
\end{proof}

\section{Regularizations}\label{regulars}
In practice, various regularizations $r_\epsilon$ are used to combat the non-invertibility of the slender body integral operator \eqref{SBT_expr}. Here we explore an example from \cite{shelley2000stokesian} which removes (at least the most obvious) invertibility issues of $(\bm{\Lambda}+\bm{K})$ and satisfies Lemma \ref{reps_small}.  \\

In \cite{shelley2000stokesian}, the integral operator $\bm{K}$ is replaced with the operator  
\begin{equation}\label{Kreg} 
\begin{aligned}
\bm{K}_\text{reg}[\bm{f}](s) &=  \frac{1}{4\pi}\big[\be_t\be_t^{\rm T}+({\bf I}+\be_t\be_t^{\rm T})\log(\pi\epsilon/4) \big]{\bm f}(s) \\
&\qquad + \frac{1}{8\pi}\int_{\T} \left(\frac{{\bf I}}{(|\bm{R}_0|^2+\epsilon^2)^{1/2}}+ \frac{\bm{R}_0\bm{R}_0^{\rm T}}{|\bm{R}_0|^2(|\bm{R}_0|^2+\epsilon^2)^{1/2}}\right){\bm f}(s') \, ds'.
\end{aligned}
\end{equation}

We have that this choice of regularization satisfies the following $\epsilon$ bound:
\begin{lemma}\label{reps_small}
For $\bm{K}$ as in \eqref{SBT_expr} and $\bm{K}_\text{reg}$ as in \eqref{Kreg}, we have that $r_\epsilon[\bm{f}](s) = \bm{K}_\text{reg}[\bm{f}](s)-\bm{K}[\bm{f}](s)$ satisfies 
\begin{equation}
\big\|r_\epsilon[\bm{f}]\big\|_{L^2(\T)} \le \epsilon\abs{\log\epsilon}C\norm{\bm{f}}_{C^1(\T)}
\end{equation}
where $C$ depends only on $c_\Gamma$ and $\kappa_{\max}$.
\end{lemma}

\begin{proof}
Denoting a point on the slender body surface $\Gamma_\epsilon$ by $\bm{R}(s,\theta)=\bm{R}_0+\epsilon\be_\rho(s,\theta)$ where $\be_\rho$ is a unit vector normal to $\X(s)$, we can write
\begin{equation}\label{r_eps}
\begin{aligned}
r_\epsilon[\bm{f}](s) &= \bm{K}_\text{reg}[\bm{f}](s)-\bm{K}[\bm{f}](s) \\
&= \bm{K}_\text{reg}[\bm{f}](s) - \bm{K}_R[\bm{f}](s,\theta) + \bm{K}_R[\bm{f}](s,\theta)-\bm{K}[\bm{f}](s)
\end{aligned}
\end{equation}
where
\begin{align*}
\bm{K}_R[\bm{f}](s,\theta) &= \frac{1}{4\pi}\big[\be_t\be_t^{\rm T}+({\bf I}+\be_t\be_t^{\rm T})\log(\pi\epsilon/4) \big]{\bm f}(s) \\
&\qquad + \frac{1}{8\pi}\int_{\T} \left(\frac{{\bf I}}{|\bm{R}|}+ \frac{\bm{R}\bm{R}^{\rm T}-\epsilon^2\be_\rho\be_\rho^{\rm T}}{|\bm{R}|^3}\right){\bm f}(s') \, ds'.
\end{align*}

Now, using Lemma 3.6 and the proof of Proposition 3.10 in \cite{closed_loop}, we have 
\begin{equation}\label{1st_bd}
\abs{\bm{K}_R-\bm{K}} \le \epsilon\abs{\log\epsilon}C\norm{\bm{f}}_{C^1(\T)}
\end{equation}
where $C$ depends on $c_\Gamma$ and $\kappa_{\max}$.\\

To estimate $\abs{\bm{K}_\text{reg}-\bm{K}_R}$, we first define
\begin{equation}\label{IR_def} I_R := \frac{1}{\abs{\bm{R}}} - \frac{1}{(\abs{\bm{R}_0}^2+\epsilon^2)^{1/2}} = \frac{2\epsilon\bm{R}_0\cdot\be_\rho}{\abs{\bm{R}}(\abs{\bm{R}_0}^2+\epsilon^2)^{1/2}((\abs{\bm{R}_0}^2+\epsilon^2)^{1/2}+\abs{\bm{R}})} .
\end{equation}
Now, since we are taking $\X\in C^3(\T)$, we can write 
\begin{equation}\label{expansion}
\bm{R}_0(s,s')=\X(s)-\X(s') = (s-s') \be_t(s) + (s-s')^2\bm{Q}(s,s'), \; \abs{\bm{Q}}\le \frac{\kappa_{\max}}{2},
\end{equation}
and therefore
\begin{equation}\label{IR_bound}
 \abs{I_R} \le C\frac{\epsilon(s-s')^2}{\abs{\bm{R}}^3}.
 \end{equation} 
 
Furthermore, we have the following inequalities: 
\begin{equation}\label{Rineq}
 \abs{\bm{R}} \ge C (\abs{\bm{R}_0}^2+\epsilon^2)^{1/2}, \quad  \abs{\bm{R}} \le \sqrt{2} (\abs{\bm{R}_0}^2+\epsilon^2)^{1/2},
 \end{equation}
 where $C$ depends on $\kappa_{\max}$ and $c_\Gamma$. To prove \eqref{Rineq}, we note that, by \eqref{expansion}, 
 \begin{align*}
  \abs{\bm{R}_0}^2+\epsilon^2 &= (s-s')^2 + 2(s-s')^3 \bm{Q}\cdot\be_{\rm s} +(s-s')^4\abs{\bm{Q}}^2 +\epsilon^2 \\
  & \le (s-s')^2 + \abs{s-s'}^3\kappa_{\max} +(s-s')^4 \frac{\kappa_{\max}}{4} +\epsilon^2\\
  & \le C((s-s')^2+\epsilon^2),
  \end{align*}
  since $\abs{s-s'}\le 1$. Then by Lemma 3.1 in \cite{closed_loop}, we have $\abs{\bm{R}} \ge C((s-s')^2+\epsilon^2)^{1/2} \ge C (\abs{\bm{R}_0}^2+\epsilon^2 )^{1/2}$. Furthermore, by Young's inequality,
 \begin{align*}
 \abs{\bm{R}}^2 = \abs{\bm{R}_0}^2+\epsilon^2 + 2\epsilon \be_\rho\cdot\bm{R}_0 \le 2(\abs{\bm{R}_0}^2+\epsilon^2).
 \end{align*}

Now we write
\begin{align*}
\bm{K}_\text{reg}-\bm{K}_R &=\frac{1}{8\pi}(\bm{J}_1+\bm{J}_2), \\
\bm{J}_1&:= \int_{\T} I_R\bigg[ {\bf I} + \bigg(\frac{1}{\abs{\bm{R}}^{2}}+ \frac{1}{\abs{\bm{R}}(\abs{\bm{R}_0}^2+\epsilon^2)^{1/2}}+ \frac{1}{\abs{\bm{R}_0}^2+\epsilon^2} \bigg)\bm{R}_0\bm{R}_0^{\rm T} \bigg]  \bm{f}(s')\, ds'\\
\bm{J}_2 &:=\int_{\T} \frac{\epsilon(\bm{R}_0\be_\rho^{\rm T}+ \be_\rho\bm{R}_0^{\rm T})}{\abs{\bm{R}}^3} \bigg) \bm{f}(s')\, ds'. 
\end{align*}

Using \eqref{Rineq}, we have
\begin{align*}
\abs{\bm{J}_1} &\le \int_{\T} C \abs{I_R} \norm{\bm{f}}_{C(\T)} \, ds' \le C\epsilon\abs{\log\epsilon}\norm{\bm{f}}_{C(\T)},
\end{align*}
where we have also used \eqref{IR_bound} and Lemma 3.3 in \cite{closed_loop}. \\

Furthermore, by \eqref{expansion} and Lemmas 3.3 and 3.4 in \cite{closed_loop}, we have
\begin{align*}
\abs{\bm{J}_2} &\le \bigg| \int_{\T} \frac{\epsilon(s-s') (\be_{\rm s}\be_\rho^{\rm T}+ \be_\rho\be_{\rm s}^{\rm T})}{\abs{\bm{R}}^3}\bm{f}(s') ds' \bigg| + C\norm{\bm{f}}_{C(\T)}\int_{\T}\frac{ \epsilon (s-s')^2}{\abs{\bm{R}}^3} ds' \\
&\le C\epsilon \abs{\log\epsilon}\norm{\bm{f}}_{C^1(\T)}.
\end{align*}

Thus $\abs{r_\epsilon[\bm{f}]}$ satisfies
\begin{align*}
\abs{r_\epsilon[\bm{f}]} &\le \abs{\bm{K}_\text{reg}-\bm{K}_R}+\abs{\bm{K}_R-\bm{K}}\le \epsilon\abs{\log\epsilon}C\norm{\bm{f}}_{C^1(\T)},
\end{align*}
and therefore
\begin{equation}\label{L2_bd_reps}
\big\|r_\epsilon[\bm{f}]\big\|_{L^2(\T)} = \bigg(\int_\T\abs{r_\epsilon[\bm{f}]}^2 ds \bigg)^{1/2}  \le \epsilon\abs{\log\epsilon}C\norm{\bm{f}}_{C^1(\T)}.
\end{equation}
\end{proof}

A solution theory for the regularized rigid slender body approximation using either the above choice of $r_\epsilon$ or any other regularization is still needed. Given that the slender body PDE framework of \cite{closed_loop,free_ends} {\it is} well-posed, we should be able to use this framework to come up with the `best' regularization for the slender body approximation. However, this is truly a deeper issue that we plan to explore in future work. 


\begin{appendix}
\section{Appendix}\label{appendix}
Here we provide proofs for the $L^2(\Gamma_\epsilon)$ trace inequality (Lemma \ref{trace2}) and the higher regularity estimate (Lemma \ref{high_reg}). \\

We first recall the following lemma, which will be used throughout the appendix. 
 \begin{lemma}\emph{(Sobolev inequality)}\label{sobolev}
Let $\Omega_{\epsilon}=\R^3\backslash\overline{\Sigma_{\epsilon}}$ be as in Section \ref{geometry}. For any $\bu\in D^{1,2}(\Omega_{\epsilon})$, we have
\begin{equation}\label{sobolev_const}
\| \bu\|_{L^6(\Omega_{\epsilon})} \le C\|\nabla\bu\|_{L^2(\Omega_{\epsilon})}
\end{equation}
where $C$ depends only on $c_\Gamma$ and $\kappa_{\max}$.
\end{lemma}
The proof of $\epsilon$-independence of $C$ appears in Appendix A.2.4 of \cite{closed_loop}. 

\subsection{Proof of Lemma \ref{trace2}}
The proof of the $L^2(\Gamma_\epsilon)$ trace inequality follows the same outline as the proof of Lemma \ref{trace1}, contained in Appendix A.2.1 of \cite{closed_loop}. In particular, using the $\epsilon$-independent $C^2$-diffeomorphisms $\psi_j$ (defined in Appendix A.2.1, \cite{closed_loop}) which map segments of the curved slender body $\Sigma_\epsilon$ to a straight cylinder, it suffices to show the $\sqrt{\epsilon}\abs{\log\epsilon}$ dependence of the trace constant for a straight cylinder. \\

Accordingly, let $\D_\rho\subset\R^2$ denote the open disk of radius $\rho$ in $\R^2$, centered at the origin, and, for some $a<\infty$, define the cylindrical surface $\Gamma_{\epsilon,a}=\p \D_\epsilon\times [-a,a]$ and the cylindrical shell $\mc{C}_{\epsilon,a}= (\D_1\backslash\overline{\D_\epsilon}) \times [-a,a]$. Consider the function space 
\[ D^{1,2}_\Gamma(\mc{C}_{\epsilon,a})= \big\{ \bu\in D^{1,2}(\mc{C}_{\epsilon,a}) \; : \; \bu\big|_{\p \mc{C}_{\epsilon,a}\backslash\Gamma_{\epsilon,a}} = 0 \big\}. \]
As in the proof of Lemma \ref{trace1}, it suffices to show the $\sqrt{\epsilon}\abs{\log\epsilon}$ dependence of the $L^2(\Gamma_{\epsilon,a})$ trace constant for functions belonging to $D^{1,2}_\Gamma(\mc{C}_{\epsilon,a})$. \\

By estimate (A.4) in \cite{closed_loop}, any $\bu\in C^1(\mc{C}_{\epsilon,a}) \cap C^0(\overline{\mc{C}_{\epsilon,a}})\cap D^{1,2}_\Gamma(\mc{C}_{\epsilon,a})$ satisfies 
\begin{align*}
\abs{{\rm Tr}(\bu)}^2 \le \abs{\log\epsilon} \int_\epsilon^1 \abs{\frac{\p\bu}{\p\rho}}^2 \rho \, d\rho.
\end{align*}
Then, noting that the surface element on $\Gamma_{\epsilon,a}$ is simply $\epsilon$, we have
\begin{align*}
\norm{{\rm Tr}(\bu)}_{L^2(\Gamma_{\epsilon,a})}^2 &= \int_{-a}^a \int_0^{2\pi} \abs{{\rm Tr}(\bu)}^2 \epsilon \, d\theta\, ds \\
&\le \epsilon\abs{\log\epsilon}  \int_{-a}^a \int_0^{2\pi}\int_\epsilon^1 \abs{\frac{\p\bu}{\p\rho}}^2 \rho \, d\rho\, d\theta\, ds  
\le \epsilon\abs{\log\epsilon}\norm{\nabla\bu}_{L^2(\mc{C}_{\epsilon,a})}^2.
\end{align*}
The same result for $\bu\in D^{1,2}_\Gamma(\mc{C}_{\epsilon,a})$ follows by density.

\subsection{Proof of Lemma \ref{high_reg}}\label{reg_lem}
To determine the $\epsilon$-dependence of the constant in \eqref{high_reg_eqn}, it suffices to work locally near the slender body surface and show that Lemma \ref{high_reg} holds within an $\epsilon$-independent region about the slender body centerline. We define the region 
\begin{equation}\label{mcO}
\mc{O} = \big\{\bx \in \Omega_{\epsilon} \; : \; \bx= \X(s) + \rho \be_{\rho}(s,\theta), \quad \epsilon < \rho<r_{\max} \big\},
\end{equation}
where $r_{\max}$ is as in Section \ref{geometry}. Within $\mc{O}$, we can use the orthonormal frame \eqref{C2_frame}. We will use the notation $\p_s,\p_\theta,\p_\rho$ to denote derivatives $\p/\p s$, $\p/\p\theta$, $\p/\p\rho$ with respect to the variables $s,\theta,\rho$, defined with respect to the orthonormal frame. We verify the $\epsilon$-dependence in the bound for $\nabla^2\bu^{\rm p}$ and $\nabla p^{\rm p}$ in two parts: we first show an $L^2$ bound for derivatives $\nabla(\p_s\bu^{\rm p})$, $\nabla(\p_\theta\bu^{\rm p})$, $\p_s p^{\rm p}$, and $\p_\theta p^{\rm p}$ in directions tangent to the slender body surface $\Gamma_\epsilon$, and then use these bounds to estimate the derivatives $\nabla(\p_\rho\bu^{\rm p})$, $\p_\rho p^{\rm p}$ normal to $\Gamma_\epsilon$. \\ 

We begin by estimating the tangential derivatives $\nabla(\p_s\bu^{\rm p})$ and $\nabla(\p_\theta\bu^{\rm p})$. Since the derivatives $\p_s$ and $\p_\theta$ with respect to the orthonormal frame \eqref{C2_frame} do not commute with the ``straight'' differential operators $\nabla$ and $\div$, we will need to make use of the following commutator bounds. 
\begin{proposition}[Commutator estimates]\label{comm_ests}
For any function $\bu\in D^{1,2}_0(\mc{O})$ and for each of the differential operators $D=\div,\, \nabla,\, \E(\cdot)$, the following commutator estimates hold:
\begin{align*}
\norm{[D,\p_\theta]\bu}_{L^2(\mc{O})} &\le C\norm{\nabla\bu}_{L^2(\mc{O})}, \quad \norm{[D,\p_s]\bu}_{L^2(\mc{O})} \le C\norm{\nabla\bu}_{L^2(\mc{O})}, 
\end{align*}
where the constant $C$ depends only on $c_\Gamma$, $\kappa_{\max}$, and $\xi_{\max}$. 
\end{proposition}

\begin{proof} 
We begin by denoting
\begin{align*}
 \be_\theta(s,\theta) &= -\sin\theta\be_{n_1}(s) + \cos\theta \be_{n_2}(s),\\
 u_\rho &= \bu\cdot\be_\rho, \; u_\theta=\bu\cdot\be_\theta, \; u_s = \bu\cdot\be_t.
 \end{align*}
 
Then, with respect to the orthonormal frame \eqref{C2_frame}, the divergence and gradient are given by 
\begin{align*}
\div \bu &= \frac{1}{1-\rho\wh\kappa} \bigg( \frac{1}{\rho} \frac{\p(\rho(1-\rho\wh\kappa)u_\rho)}{\p\rho} + \frac{1}{\rho}\frac{\p ((1-\rho\wh\kappa)u_\theta)}{\p\theta} + \frac{\p u_s}{\p s} \bigg) \\
\nabla \bu &= \be_\rho(s,\theta)\frac{\p\bu}{\p\rho}^{\rm T} + \be_\theta(s,\theta)\frac{1}{\rho}\frac{\p\bu}{\p\theta}^{\rm T} + \be_t(s)\frac{1}{1-\rho\wh\kappa} \bigg(\frac{\p \bu}{\p s} -\kappa_3 \frac{\p \bu}{\p\theta} \bigg)^{\rm T},
\end{align*}
where 
\begin{equation}\label{kappahat}
\wh\kappa(s,\theta) = \kappa_1(s)\cos\theta + \kappa_2(s)\sin\theta.
\end{equation}

Direct computation of the commutators yields 
\begin{align*}
[\div,\p_\theta]\bu &= \frac{(\p_\theta\wh\kappa)}{1-\rho\wh\kappa}\bigg( \rho\,\div\bu - \frac{1}{\rho}\frac{\p}{\p\rho} \big(\rho^2 u_\rho \big)  - \frac{\p u_\theta}{\p\theta} \bigg) - \frac{(\p_\theta^2\wh\kappa)}{1-\rho\wh\kappa}u_\theta  \\
[\div,\p_s]\bu &= \frac{(\p_s\wh\kappa)}{1-\rho\wh\kappa}\bigg( \rho\,\div\bu - \frac{1}{\rho}\frac{\p}{\p\rho} \big(\rho^2 u_\rho \big)  - \frac{\p u_\theta}{\p\theta} \bigg) - \frac{(\p_\theta\p_s\wh\kappa)}{1-\rho\wh\kappa}u_\theta  \\
[\nabla,\p_\theta]\bu &= \be_\theta\frac{\p\bu}{\p\rho}^{\rm T} - \be_\rho\frac{1}{\rho}\frac{\p\bu}{\p\theta}^{\rm T} + \be_t\frac{\rho(\p_\theta\wh\kappa)}{(1-\rho\wh\kappa)^2} \bigg(\frac{\p \bu}{\p s} -\kappa_3 \frac{\p \bu}{\p\theta} \bigg)^{\rm T} \\
[\nabla,\p_s]\bu &= (\p_s\be_\rho)\frac{\p\bu}{\p\rho}^{\rm T} + (\p_s\be_\theta)\frac{1}{\rho}\frac{\p\bu}{\p\theta}^{\rm T} + \bigg(\be_t \frac{\rho(\p_s\wh\kappa)}{1-\rho\wh\kappa}+(\p_s\be_t) \bigg)\frac{1}{1-\rho\wh\kappa} \bigg(\frac{\p \bu}{\p s} -\kappa_3 \frac{\p \bu}{\p\theta} \bigg)^{\rm T} 
\end{align*}

Using \eqref{kappahat} and the orthonormal frame ODEs \eqref{C2_frame}, we have
\begin{align*}
\abs{\p_\theta \wh\kappa} &= \abs{-\kappa_1\sin\theta+  \kappa_2\cos\theta}\le \kappa_{\max}, \quad \abs{\p_s \wh\kappa} = \abs{\kappa_1'\cos\theta+\kappa_2'\sin\theta}\le \xi_{\max} + 2(\kappa_{\max}+\pi), \\
\abs{\p_\theta^2\wh\kappa} &= \abs{-\wh\kappa}\le \kappa_{\max}, \quad \abs{\p_\theta\p_s \wh\kappa} = \abs{-\kappa_1'\sin\theta+  \kappa_2'\cos\theta}\le \xi_{\max}+ 2(\kappa_{\max}+\pi),  \\
\abs{\p_s \be_\rho} &= \abs{-\wh\kappa\be_t + \kappa_3\be_\theta}\le \kappa_{\max} +\pi, \quad \abs{\p_s\be_\theta} = \abs{-(\p_\theta \wh\kappa)\be_t -\kappa_3\be_t}\le \kappa_{\max} +\pi, \\
\abs{\frac{1}{1-\rho\wh\kappa} } &\le \frac{1}{1- r_{\max}\kappa_{\max}(\cos\theta+\sin\theta) } \le \frac{1}{1- \frac{1}{2\kappa_{\max}}\kappa_{\max}\sqrt{2} } \le 4 .
\end{align*}

Finally, noting that, by Lemma \ref{sobolev},
\[ \norm{u_\theta}_{L^2(\mc{O})} \le \abs{\mc{O}}^{1/3} \norm{\bu}_{L^6(\mc{O})} \le C\norm{\nabla\bu}_{L^2(\mc{O})}, \]
the desired $L^2(\Omega)$ bounds follow for each of $D=\div,\nabla$. The estimate for the symmetric gradient $\E(\bu)$ then follows from the gradient commutator bound. 
\end{proof}

 
Now, to derive an estimate for $\nabla(\p_s\bu^{\rm p})$, we will make use of Definition \ref{rigid_weak_p} with a particular test function $\bm{\varphi}$, which we will construct here. First, we want our test function to be supported only within $\mc{O}$. We define a smooth cutoff function 
\begin{equation}\label{Ocutoff}
\psi(\rho) = \begin{cases}
1, & \rho < r_{\max}/4 \\
0, & \rho > r_{\max}/2, 
\end{cases} \quad \abs{\frac{\p\psi}{\p\rho} } \le C,
\end{equation}
where $C$ depends only on $r_{\max}$. Note that $\psi(\rho)$ commutes with both $\p_\theta$ and $\p_s$. \\

We would like to use $\p_s^2(\psi\bu^{\rm p})$ as a test function in Definition \ref{rigid_weak_p}, but it will be more convenient to work with a function which vanishes on $\Gamma_\epsilon$. We therefore construct a correction $\bg \in C^2(\Omega_\epsilon)$ supported only in $\mc{O}$ and satisfying
\begin{equation}\label{g_correction}
\bg\big|_{\Gamma_\epsilon} = (\p_s \bu^{\rm p})\big|_{\Gamma_\epsilon} = \bm{\omega}^{\rm p}\times\be_t(s), \quad \norm{\nabla \bg}_{L^2(\mc{O})} \le C\abs{\bm{\omega}},
\end{equation}
where $C$ depends on $c_\Gamma$ and $\kappa_{\max}$. To build $\bg$, we follow a similar construction used in Section 4.1 of \cite{closed_loop}. We define
\[ \bg_0(\rho,\theta,s) = \begin{cases}
\bm{\omega}^{\rm p}\times \be_t(s) & \text{if } \rho<4\epsilon \\
0 & \text{otherwise}
\end{cases} \] 
and take
\[ \bg(\rho,\theta,s):= \phi_\epsilon(\rho)\bg_0(\rho,\theta,s), \]
where $\phi_\epsilon(\rho)$ is the smooth cutoff defined in \eqref{cutoff_def}-\eqref{phi_decay}. Note that $\bg\in C^2$ and is supported within the region
\[ \mc{O}_\epsilon := \big\{ \X(s) + \rho \be_\rho(s,\theta) \; : \; s\in \T, \, \epsilon \le \rho \le 4\epsilon, \, 0\le \theta <2\pi \big\}, \]
where $\abs{\mc{O}_\epsilon} \le C\epsilon^2$. Then, using \eqref{phi_decay} and \eqref{C2_frame}, we have
\begin{align*} 
\norm{\nabla \bg}_{L^2(\mc{O})} &\le \sqrt{\abs{\mc{O}_\epsilon}}\norm{\nabla \bg}_{C(\mc{O}_\epsilon)} \\
& \le \sqrt{\abs{\mc{O}_\epsilon}}\bigg( \norm{\frac{\p\phi_\epsilon}{\p\rho}}_{C(\mc{O}_\epsilon)}\norm{\bg_0}_{C(\mc{O}_\epsilon)}+ \norm{\frac{1}{1-\rho\wh\kappa}\frac{\p\bg_0}{\p s}}_{C(\mc{O}_\epsilon)} \bigg) \le C\abs{\bm{\omega}^{\rm p}}.
\end{align*}

Now, we could just use $\p_s(\p_s(\psi\bu^{\rm p}) -\bg)$ as a test function in Definition \ref{rigid_weak_p}, but it will actually be useful to include a second correction term in the following way. We consider $\bz\in D^{1,2}_0(\mc{O})$ satisfying
\begin{equation}\label{zee_def}
\begin{aligned}
\div\bz &= \div(\psi\p_s\bu^{\rm p} -\bg) \quad \text{in } \mc{O} \\
\norm{\nabla \bz}_{L^2(\mc{O})} &\le C\norm{\div(\psi\p_s\bu^{\rm p} -\bg)}_{L^2(\mc{O})} 
\end{aligned}
\end{equation}
for $C$ depending only on $c_\Gamma$ and $\kappa_{\max}$. We know that such a $\bz$ exists due to \cite{galdi2011introduction}, Section III.3, and the constant $C$ is independence of $\epsilon$ due to Appendix A.2.5 of \cite{closed_loop}. Furthermore, since $\div\bu^{\rm p}=0$, by Proposition \ref{comm_ests} we have  
\begin{align*}
\norm{\div(\psi\p_s\bu^{\rm p}-\bg)}_{L^2(\mc{O})} &\le \norm{\div(\p_\theta\bu^{\rm p})}_{L^2(\mc{O})}+ C\norm{\p_\theta\bu^{\rm p}}_{L^2(\mc{O})}+\norm{\nabla \bg}_{L^2(\mc{O})} \\
 &\le \norm{[\div,\p_\theta]\bu^{\rm p}}_{L^2(\mc{O})} + C\norm{\p_\theta\bu^{\rm p}}_{L^2(\mc{O})} + C\abs{\bm{\omega}^{\rm p}} \\
 &\le C\norm{\nabla\bu^{\rm p}}_{L^2(\mc{O})} + C\abs{\bm{\omega}^{\rm p}}.
\end{align*}
Here we have also used that $\norm{\p_\theta\bu^{\rm p}}_{L^2(\mc{O})} \le \norm{\rho \nabla\bu^{\rm p}}_{L^2(\mc{O})} \le r_{\max}\norm{\nabla\bu^{\rm p}}_{L^2(\mc{O})}$. In particular, $\bz$ satisfying \eqref{zee_def} also satisfies
\begin{equation}\label{zestimate}
\norm{\nabla\bz}_{L^2(\mc{O})} \le C\norm{\nabla\bu^{\rm p}}_{L^2(\mc{O})}  + C\abs{\bm{\omega}^{\rm p}}.
\end{equation}

Using extension by zero to consider $\bz$ as a function over all $\Omega_\epsilon$, we can now construct our desired test function for use in Definition \ref{rigid_weak_p}. In particular, we will use the function $\p_s(\p_s(\psi\bu^{\rm p})-\bg-\bz)$ in place of $\bm{\varphi}$ in Definition \ref{rigid_weak_p}. Note that by definition of $\bz$, this function may only belong to $L^2(\Omega_\epsilon)$. In this case, we can make sense of the following integration-by-parts argument using finite differences rather than full derivatives (see \cite{boyer2012mathematical}, Section III.2.7 for construction of finite difference operators along a curved boundary). Thus we really only need $\p_s(\psi\bu^{\rm p})-\bg-\bz\in D^{1,2}(\Omega_\epsilon)$ to make sense of the following result. Note that in integrating by parts, we will also need to make use of the fact that, for $i=s,\theta$,
\begin{equation}\label{jacfac_i}
 \p_i (d\bx) = -\frac{\rho\p_i\wh\kappa}{1-\rho\wh\kappa} d\bx := \mc{J}_i \, d\bx, \quad \abs{\mc{J}_i} \le C; \quad i=s,\theta ,
 \end{equation}
where $C$ depends on $c_\Gamma$, $\kappa_{\max}$, and $\xi_{\max}$. \\

Then, using $\p_s(\p_s(\psi\bu^{\rm p})-\bg-\bz)$ in Definition \ref{rigid_weak_p}, we have 
\begin{align*}
0 &= \int_{\mc{O}} \bigg(2\E(\bu^{\rm p}): \E\big(\p_s(\p_s(\psi\bu^{\rm p})-\bg -\bz) \big) - p^{\rm p} \,\div(\p_s(\p_s(\psi\bu^{\rm p})-\bg -\bz)) \bigg) \, d\bx \\
&= \int_{\mc{O}} 2\E(\bu^{\rm p}): \p_s\E(\p_s(\psi\bu^{\rm p})-\bg-\bz) \, d\bx + \int_{\mc{O}} 2\E(\bu^{\rm p}): [\E(\cdot),\p_s](\p_s(\psi\bu^{\rm p})-\bg -\bz) \, d\bx \\
&\qquad - \int_{\mc{O}} p^{\rm p} \,\p_s(\div(\p_s(\psi\bu^{\rm p})-\bg -\bz)) \, d\bx - \int_{\mc{O}} p^{\rm p} \,[\div,\p_s](\p_s(\psi\bu^{\rm p})- \bg -\bz) \, d\bx \\
&= -\int_{\mc{O}} 2\p_s\E(\bu^{\rm p}): \E(\p_s(\psi\bu^{\rm p})-\bg -\bz) \, d\bx -\int_{\mc{O}} 2\E(\bu^{\rm p}): \E(\p_s(\psi\bu^{\rm p})-\bg -\bz) \, \mc{J}_s \, d\bx \\
&\qquad + \int_{\mc{O}} 2\E(\bu^{\rm p}): [\E(\cdot),\p_s](\p_s(\psi\bu^{\rm p})-\bg -\bz) \, d\bx - \int_{\mc{O}} p^{\rm p} \,[\div,\p_s](\p_s(\psi\bu^{\rm p})- \bg -\bz) \, d\bx \\
&= -\int_{\mc{O}} 2\E(\p_s\bu^{\rm p}):\E(\p_s(\psi\bu^{\rm p})-\bg -\bz) \, d\bx - \int_{\mc{O}} 2\E(\bu^{\rm p}): \E(\p_s(\psi\bu^{\rm p})-\bg -\bz) \, \mc{J}_s \, d\bx \\
&\qquad +\int_{\mc{O}} 2[\E(\cdot),\p_s]\bu^{\rm p}: \E(\p_s(\psi\bu^{\rm p})- \bg -\bz) \, d\bx + \int_{\mc{O}} 2\E(\bu^{\rm p}): [\E(\cdot),\p_s](\p_s(\psi\bu^{\rm p})-\bg -\bz) \, d\bx \\
&\qquad  - \int_{\mc{O}} p^{\rm p} \,[\div,\p_s](\p_s(\psi\bu^{\rm p})- \bg -\bz) \, d\bx.
\end{align*}
Note that the first integral in the third line vanishes due to the definition of $\bz$. In this way we we can avoid having to deal with a $\p_s p^{\rm p}$ term in the resulting estimate.  \\

Then, using Proposition \ref{comm_ests}, estimates \eqref{zestimate} and \eqref{g_correction}, and Lemma \ref{korn}, we have 
\begin{align*}
\norm{\E(\p_s\bu^{\rm p})}_{L^2(\mc{O})}^2 &\le  C\norm{\E(\p_s\bu^{\rm p})}_{L^2(\mc{O})} \big( \norm{\p_s \bu^{\rm p}}_{L^2(\mc{O})} + \norm{\E(\bz)}_{L^2(\mc{O})} + \norm{\E(\bg)}_{L^2(\mc{O})} \big) \\
&\quad + C\norm{\E(\bu^{\rm p})}_{L^2(\mc{O})} \big(\norm{\E(\psi\p_s\bu^{\rm p})}_{L^2(\mc{O})} +\norm{\E(\bg)}_{L^2(\mc{O})} +\norm{\E(\bz)}_{L^2(\mc{O})} \big) \\
&\quad + 2\norm{[\E(\cdot),\p_s]\bu^{\rm p}}_{L^2(\mc{O})}\big( \norm{\E(\psi\p_s\bu^{\rm p})}_{L^2(\mc{O})} + \norm{\E(\bz)}_{L^2(\mc{O})} + \norm{\E(\bg)}_{L^2(\mc{O})} \big) \\
&\quad +2\norm{\E(\bu^{\rm p})}_{L^2(\mc{O})} \big( \norm{[\E(\cdot),\p_s](\psi\p_s\bu^{\rm p})}_{L^2(\mc{O})}  + \norm{[\E(\cdot),\p_s](\bz)}_{L^2(\mc{O})} + \norm{[\E(\cdot),\p_s](\bg)}_{L^2(\mc{O})} \big) \\
&\quad + \norm{ p^{\rm p}}_{L^2(\mc{O})} \big( \norm{[\div,\p_s](\psi\p_s\bu^{\rm p})}_{L^2(\mc{O})} + \norm{[\div,\p_s](\bz)}_{L^2(\mc{O})} +\norm{[\div,\p_s](\bg)}_{L^2(\mc{O})} \big)  \\
&\le C(\norm{\E(\p_s\bu^{\rm p})}_{L^2(\mc{O})} + \norm{\nabla\bu^{\rm p}}_{L^2(\mc{O})} +\abs{\bm{\omega}})\big(\norm{\nabla\bu^{\rm p}}_{L^2(\mc{O})}+\norm{p^{\rm p}}_{L^2(\mc{O})} +\abs{\bm{\omega}^{\rm p}} \big) \\
&\le \delta\norm{\E(\p_s\bu^{\rm p})}_{L^2(\mc{O})}^2 + C(\delta) \big(\norm{\nabla\bu^{\rm p}}_{L^2(\mc{O})}^2+\norm{p^{\rm p}}_{L^2(\mc{O})}^2 + \abs{\bm{\omega}^{\rm p}}^2\big)
\end{align*}
for any $0<\delta\in\R$, by Young's inequality. Taking $\delta=\frac{1}{2}$ and using Lemma \ref{korn}, we obtain
\begin{equation}\label{ps_est}
\begin{aligned}
\norm{\nabla(\p_s\bu^{\rm p})}_{L^2(\mc{O})} &\le \norm{\E(\p_s\bu^{\rm p})}_{L^2(\mc{O})} \le C \big(\norm{\nabla\bu^{\rm p}}_{L^2(\mc{O})}+\norm{p^{\rm p}}_{L^2(\mc{O})} + \abs{\bm{\omega}^{\rm p}} \big) \\
& \le C\abs{\log\epsilon}^{1/2} \big(\norm{\nabla\bu^{\rm p}}_{L^2(\Omega_\epsilon)}+\norm{p^{\rm p}}_{L^2(\Omega_\epsilon)} \big),
\end{aligned}
\end{equation}
where we have used Corollary \ref{omegaP} to bound $\abs{\bm{\omega}^{\rm p}}$. Here $C$ depends only on $c_\Gamma$, $\kappa_{\max}$, and $\xi_{\max}$. \\

We may estimate $\p_\theta\bu^{\rm p}$ in a similar way. In fact, the construction of the analogous test function is simpler since $(\p_\theta\bu^{\rm p})\big|_{\Gamma_\epsilon} = \p_\theta(\bv+\bm{\omega}\times \X(s))=0$ and thus we do not need to correct for a nonzero boundary value. Following the same steps used to estimate $\p_s\bu^{\rm p}$, we obtain
\begin{equation}\label{ptheta_est}
\norm{\nabla(\p_\theta\bu^{\rm p})}_{L^2(\mc{O})} \le C \big(\norm{\nabla\bu^{\rm p}}_{L^2(\Omega_\epsilon)}+\norm{p^{\rm p}}_{L^2(\Omega_\epsilon)} \big),
\end{equation}
where $C$ depends only on $c_\Gamma$, $\kappa_{\max}$, and $\xi_{\max}$. \\

In addition to the estimates \eqref{ps_est} and \eqref{ptheta_est}, we need bounds for the tangential derivatives $\p_s p^{\rm p}$ and $\p_\theta p^{\rm p}$ of the pressure. We begin by estimating $\p_s p^{\rm p}$; the bound for $\p_\theta p^{\rm p}$ is similar. Since we already know that $\p_s p^{\rm p}\in L^2(\Omega_\epsilon)$, we may consider $\wt\bz \in D^{1,2}_0(\mc{O})$ satisfying
\begin{equation}\label{wtzee_def}
\begin{aligned}
\div\wt\bz &= \psi \p_s p^{\rm p} \quad \text{in }\mc{O}, \\
\norm{\nabla\wt\bz}_{L^2(\mc{O})} &\le C\norm{\psi \p_s p^{\rm p}}_{L^2(\mc{O})},
\end{aligned}
\end{equation}
where $\psi$ is as in \eqref{Ocutoff}. Again, we know that such a $\wt\bz$ exists due to \cite{galdi2011introduction}, Section III.3 and \cite{closed_loop}, Appendix A.2.5. \\

Using $\p_s\wt\bz$ as a test function in Definition \ref{rigid_weak_p} (again, we can make sense of the following computation using finite differences, and thus only require $\wt\bz\in D^{1,2}(\mc{O})$), we have
\begin{align*}
0&= \int_{\mc{O}} \bigg(2\E(\bu^{\rm p}): \E(\p_s \wt\bz) - p^{\rm p}\, \div(\p_s\wt\bz)\bigg) \, d\bx = \int_{\mc{O}}2\E(\bu^{\rm p}): \p_s\E(\wt\bz) \, d\bx \\
&\quad +  \int_{\mc{O}}2\E(\bu^{\rm p}): [\E(\cdot),\p_s]\wt\bz \, d\bx - \int_{\mc{O}} p^{\rm p} \,\p_s\div\wt\bz \, d\bx  - \int_{\mc{O}} p^{\rm p} \, [\div,\p_s]\wt\bz \, d\bx \\
&= -\int_{\mc{O}}2\p_s\E(\bu^{\rm p}): \E(\wt\bz) \, d\bx - \int_{\mc{O}}2\E(\bu^{\rm p}): \E(\wt\bz) \, \mc{J}_s \, d\bx +  \int_{\mc{O}}2\E(\bu^{\rm p}): [\E(\cdot),\p_s]\wt\bz \, d\bx \\
&\quad - \int_{\mc{O}} p^{\rm p} \, [\div,\p_s]\wt\bz \, d\bx  + \int_{\mc{O}}(\p_s p)\div\wt\bz \, d\bx + \int_{\mc{O}} p^{\rm p} \,\div\wt\bz \, \mc{J}_s \, d\bx \\
&= \int_{\mc{O}}\psi(\p_s p)^2 \, d\bx -\int_{\mc{O}}2\p_s\E(\bu^{\rm p}): \E(\wt\bz) \, d\bx - \int_{\mc{O}}2\E(\bu^{\rm p}): \E(\wt\bz) \, \mc{J}_s \, d\bx \\
&\quad  +  \int_{\mc{O}}2\E(\bu^{\rm p}): [\E(\cdot),\p_s]\wt\bz \, d\bx - \int_{\mc{O}} p^{\rm p} \, [\div,\p_s]\wt\bz \, d\bx  + \int_{\mc{O}} p^{\rm p} \,\div\wt\bz \, \mc{J}_s \, d\bx,
\end{align*}
where $\mc{J}_s \, d\bx$ is as in \eqref{jacfac_i} and we have used \eqref{wtzee_def}. Then, using that $\psi^2\le \psi$, we have
\begin{align*}
\norm{\psi\p_s p^{\rm p}}_{L^2(\mc{O})}^2 &\le 2\norm{\E(\p_s\bu^{\rm p})}_{L^2(\mc{O})} \norm{\E(\wt\bz)}_{L^2(\mc{O})} + 2\norm{[\E(\cdot),\p_s]\bu^{\rm p}}_{L^2(\mc{O})} \norm{\E(\wt\bz)}_{L^2(\mc{O})} \\
&\quad + C\norm{\E(\bu^{\rm p})}_{L^2(\mc{O})} \norm{\E(\wt\bz)}_{L^2(\mc{O})} + 2\norm{\E(\bu^{\rm p})}_{L^2(\mc{O})}\norm{ [\E(\cdot),\p_s]\wt\bz}_{L^2(\mc{O})} \\
&\quad +\norm{ p^{\rm p}}_{L^2(\mc{O})} \norm{ [\div,\p_s]\wt\bz}_{L^2(\mc{O})} + C\norm{p^{\rm p}}_{L^2(\mc{O})} \norm{\div\wt\bz}_{L^2(\mc{O})}\\
&\le C\big(\norm{\nabla(\p_s\bu^{\rm p})}_{L^2(\mc{O})} + \norm{\nabla\bu^{\rm p}}_{L^2(\mc{O})} +\norm{ p^{\rm p}}_{L^2(\mc{O})} \big) \norm{ \psi\p_s p^{\rm p}}_{L^2(\mc{O})} \\
&\le \delta \norm{ \psi\p_s p^{\rm p}}_{L^2(\mc{O})}^2 + C(\delta)\big(\norm{\nabla(\p_s\bu^{\rm p})}_{L^2(\mc{O})}^2 + \norm{\nabla\bu^{\rm p}}_{L^2(\mc{O})}^2 +\norm{ p^{\rm p}}_{L^2(\mc{O})}^2 \big)
\end{align*}
for $0<\delta\in \R$. Here we have used \eqref{jacfac_i}, \eqref{wtzee_def}, Proposition \ref{comm_ests}, and Young's inequality. Taking $\delta=\frac{1}{2}$ and using \eqref{ps_est}, we obtain
\[ \norm{\psi\p_s p^{\rm p}}_{L^2(\mc{O})} \le C\abs{\log\epsilon}^{1/2}\big( \norm{\nabla\bu^{\rm p}}_{L^2(\Omega_\epsilon)} +\norm{ p^{\rm p}}_{L^2(\Omega_\epsilon)} \big). \]
Then, using \eqref{Ocutoff}, within the region
\[ \mc{O}' =  \bigg\{\bx \in \Omega_{\epsilon} \; : \; \bx= \X(s) + \rho \be_{\rho}(s,\theta), \quad \epsilon < \rho<\frac{r_{\max}}{4} \bigg\}, \]
we have 
\begin{equation}\label{psp_est}
\norm{\p_s p^{\rm p}}_{L^2(\mc{O}')} \le C\abs{\log\epsilon}^{1/2}\big( \norm{\nabla\bu^{\rm p}}_{L^2(\Omega_\epsilon)} +\norm{ p^{\rm p}}_{L^2(\Omega_\epsilon)} \big)
\end{equation}
for $C$ depending only on $c_\Gamma$, $\kappa_{\max}$, and $\xi_{\max}$. \\

We can similarly use \eqref{ptheta_est} to show
\begin{equation}\label{pthetap_est}
\norm{\p_\theta p^{\rm p}}_{L^2(\mc{O}')} \le C\big( \norm{\nabla\bu^{\rm p}}_{L^2(\Omega_\epsilon)} +\norm{ p^{\rm p}}_{L^2(\Omega_\epsilon)} \big).
\end{equation}


Now we can use the tangential bounds \eqref{ps_est}, \eqref{ptheta_est}, \eqref{psp_est}, and \eqref{pthetap_est} to obtain an estimate for derivatives $\nabla(\p_\rho\bu^{\rm p})$ normal to $\Gamma_\epsilon$. For this, we will use the full Stokes equations \eqref{SB_PDE}, written with respect to the orthonormal frame $\be_t$, $\be_\rho$, $\be_\theta$ in $\mc{O}$ as
\begin{align*}
- \Delta \bu^{\rm p} +\nabla p^{\rm p} &= -\Delta \bu^{\rm p} + \frac{\p p^{\rm p}}{\p \rho}\be_{\rho} + \frac{1}{\rho}\frac{\p p^{\rm p}}{\p\theta}\be_{\theta} +
\frac{1}{1-\rho\wh\kappa}\bigg(\frac{\p p^{\rm p}}{\p s}-\kappa_3\frac{\p p^{\rm p}}{\p\theta}\bigg)\be_t =0 \\
\div \bu^{\rm p} &= \frac{1}{1-\rho\wh\kappa}\bigg(\frac{1}{\rho}\frac{\p(\rho (1-\rho\wh\kappa) u_{\rho})}{\p \rho}+\frac{1}{\rho}\frac{\p ((1-\rho\wh\kappa) u_{\theta})}{\p\theta} + \frac{\p u_s}{\p s} \bigg) = 0.
\end{align*}
Here $\wh\kappa$ is as in \eqref{kappahat} and we recall the notation $u_\rho = \bu^{\rm p}\cdot\be_\rho$, $u_\theta=\bu^{\rm p}\cdot\be_\theta$, $u_s=\bu^{\rm p}\cdot\be_t$. \\

From the divergence-free condition on $\bu^{\rm p}$, after multiplying through by $\rho (1-\rho\wh\kappa)$ and differentiating once with respect to $\rho$, we obtain
\begin{align*}
\norm{ \frac{\p^2 u_{\rho}}{\p^2 \rho}}_{L^2(\mathcal{O})} &\le C\bigg(\norm{\frac{1}{\rho} \nabla \bu^{\rm p} }_{L^2(\mathcal{O})} + \bigg\|\frac{1}{\rho}\bigg\|_{L^{\infty}(\mathcal{O})}|\mathcal{O}|^{1/3}\big\|\bu^{\rm p} \big\|_{L^6(\mathcal{O})} \\
&\hspace{2cm}+ \bigg\|\frac{1}{\rho}\frac{\p}{\p \rho}\bigg(\frac{\p u_{\theta}}{\p\theta}\bigg)\bigg\|_{L^2(\mathcal{O})} +\bigg\|\frac{\p}{\p \rho}\bigg(\frac{\p u_s}{\p s}\bigg)\bigg\|_{L^2(\mathcal{O})} \bigg)\\
&\le \frac{C}{\epsilon}\abs{\log\epsilon}^{1/2}\big( \|\nabla \bu^{\rm p}\|_{L^2(\Omega_\epsilon)} +\norm{ p^{\rm p}}_{L^2(\Omega_\epsilon)} \big),
\end{align*}
where we have used \eqref{ps_est} and \eqref{ptheta_est} along with the Sobolev inequality on $\Omega_\epsilon$. \\

Furthermore, using the $\be_\rho$ component of $-\Delta \bu^{\rm p} +\nabla p=0$, we have 
\begin{align*} 
\frac{\p p^{\rm p}}{\p \rho} &= (\Delta\bu^{\rm p}) \cdot\be_{\rho} \\
&= \frac{1}{\rho (1-\rho\wh\kappa)}\frac{\p}{\p \rho}\left(\rho (1-\rho\wh\kappa)\frac{\p \bu^{\rm p}}{\p \rho}\right)\cdot\be_{\rho} +\frac{1}{\rho^2(1-\rho\wh\kappa)}\frac{\p}{\p \theta}\bigg((1-\rho\wh\kappa)\frac{\p \bu^{\rm p}}{\p\theta}\bigg)\cdot\be_{\rho} \\
&\hspace{4cm} +\frac{1}{1-\rho\wh\kappa} \frac{\p}{\p s}\bigg( \frac{1}{1-\rho\wh\kappa}\bigg[ \frac{\p \bu^{\rm p}}{\p s}- \kappa_3\frac{\p \bu^{\rm p}}{\p \theta} \bigg] \bigg)\cdot\be_{\rho} \\
&= \frac{1}{\rho (1-\rho\wh\kappa)}\frac{\p}{\p \rho}\left(\rho (1-\rho\wh\kappa)\frac{\p u_{\rho}}{\p \rho}\right) +\frac{1}{\rho^2(1-\rho\wh\kappa)}\frac{\p}{\p \theta}\bigg((1-\rho\wh\kappa)\frac{\p \bu^{\rm p}}{\p\theta}\bigg)\cdot\be_{\rho} \\
&\hspace{4cm} +\frac{1}{1-\rho\wh\kappa} \frac{\p}{\p s}\bigg( \frac{1}{1-\rho\wh\kappa}\bigg[ \frac{\p \bu^{\rm p}}{\p s}- \kappa_3\frac{\p \bu^{\rm p}}{\p \theta} \bigg] \bigg)\cdot\be_{\rho}, 
\end{align*}
since $\be_{\rho}(s,\theta)$ does not vary with $\rho$. Therefore, using \eqref{ps_est}, \eqref{ptheta_est}, \eqref{psp_est}, and \eqref{pthetap_est}, along with the the bound on $\frac{\p^2 u_{\rho}}{\p \rho^2}$, we have
\[ \|\nabla p^{\rm p}\|_{L^2(\mathcal{O}')} \le \frac{C}{\epsilon}\abs{\log\epsilon}^{1/2}\big( \|\nabla \bu^{\rm p}\|_{L^2(\Omega_\epsilon)} +\norm{ p^{\rm p}}_{L^2(\Omega_\epsilon)} \big). \]

Finally, to estimate $\frac{\p^2 u_j}{\p \rho^2}$, $j=\theta,s$, we again use that 
\begin{align*} 
 \nabla p^{\rm p}\cdot\be_j &= (\Delta \bu^{\rm p})\cdot \be_j(s,\theta) \\
&= \frac{1}{\rho (1-\rho\wh\kappa)}\frac{\p}{\p \rho}\bigg(\rho (1-\rho\wh\kappa)\frac{\p u_j}{\p \rho}\bigg) +\frac{1}{\rho^2(1-\rho\wh\kappa)}\frac{\p}{\p \theta}\bigg((1-\rho\wh\kappa)\frac{\p \bu^{\rm p}}{\p\theta}\bigg)\cdot\be_j  \\
&\hspace{4cm}+\frac{1}{1-\rho\wh\kappa} \frac{\p}{\p s}\bigg( \frac{1}{1-\rho\wh\kappa}\bigg[ \frac{\p \bu^{\rm p}}{\p s}- \kappa_3\frac{\p \bu^{\rm p}}{\p \theta} \bigg] \bigg)\cdot\be_j , \quad j=\theta,s,
\end{align*}
since each of $\be_t(s)$, $\be_\rho(s,\theta)$ and $\be_\theta(s,\theta)$ are independent of $\rho$. Then we have
\begin{align*}
 \bigg\|\frac{\p^2 u_j}{\p \rho^2}\bigg\|_{L^2(\mathcal{O}')} &\le C \bigg( \bigg\|\frac{1}{\rho}\bigg\|_{L^{\infty}(\mathcal{O}')} \| \nabla \bu^{\rm p}\|_{L^2(\mathcal{O}')} + \bigg\|\frac{\p^2\bu^{\rm p}}{\p s^2}\bigg\|_{L^2(\mathcal{O}')} \\
 &\hspace{2cm} + \bigg\|\frac{\p^2\bu^{\rm p}}{\p s\p\theta}\bigg\|_{L^2(\mathcal{O}')} +\bigg\|\frac{\p^2\bu^{\rm p}}{\p \theta^2}\bigg\|_{L^2(\mathcal{O}')}+\|\nabla p^{\rm p}\|_{L^2(\mathcal{O}')} \bigg) \\
&\le  \frac{C}{\epsilon}\abs{\log\epsilon}^{1/2} \big(\norm{\nabla\bu^{\rm p}}_{L^2(\Omega_\epsilon)} + \norm{p^{\rm p}}_{L^2(\Omega_\epsilon)} \big), \quad j=\theta, s, 
 \end{align*}
 where $C$ depends only on $c_\Gamma$, $\kappa_{\max}$, and $\xi_{\max}$. Altogether, we obtain Lemma \ref{high_reg}. \hfill $\square$
 
 \begin{remark}  
We note that the factor of $\frac{1}{\epsilon}$ in Lemma \ref{high_reg} is necessary. As a heuristic, we consider an infinite straight cylinder of radius $\epsilon$ and take $\bu = (\frac{1}{\rho}-\frac{1}{\epsilon}) \be_{\theta}$, where $\be_{\theta}$ is now the (constant) angular vector in straight cylindrical coordinates, and $p\equiv$ constant. Ignoring decay conditions toward infinity along the cylinder, $(\bu,p)$ solves the Stokes equations with $\bu=0$ on the cylinder surface. Then
\begin{align*}
|\nabla^2\bu| =\bigg| \frac{\p^2}{\p \rho^2} \frac{1}{\rho}\bigg| = \bigg|\frac{2}{\rho^3}\bigg| = \frac{2}{\rho}\big| \nabla \bu \big|,
\end{align*}
and within the region $\epsilon < \rho \le 2 \epsilon$, we have $|\nabla^2\bu| \ge \frac{1}{\epsilon}|\nabla \bu|$. 
\end{remark}

\end{appendix}

\bibliography{rigid_bib.bib}{}
\bibliographystyle{abbrv}


\end{document}